\documentclass[10pt]{article}
\usepackage{latexsym} 
\usepackage[all]{xy}
\usepackage{amsfonts} 
\usepackage{color}
\usepackage{cite}
\usepackage{authblk}
\usepackage{amsmath}
\usepackage{mathrsfs}
\usepackage{extarrows}
\usepackage{amssymb}
\usepackage{tikz-cd}
\usepackage{appendix}
\usetikzlibrary{calc}
\usepackage{stmaryrd}
\usepackage[utf8]{inputenc}
\usepackage{geometry}
\usepackage{amsthm}
\usepackage{todonotes}
\usepackage{eucal}
\usepackage{fullpage}
\usepackage[colorlinks=true,linkcolor=black,citecolor=blue,urlcolor=blue]{hyperref}

\title{\textbf{Deformation Theory of $\mathbb{E}_n$-Monoidal Categories}}
\author{Yining Chen}
\affil{}

\begin{document}

\maketitle

\newtheorem{definition}{Definition}[section]
\newtheorem{theorem}[definition]{Theorem}
\newtheorem*{theorem*}{Theorem}
\newtheorem{prop}[definition]{Proposition}
\newtheorem{lemma}[definition]{Lemma}
\newtheorem{remark}[definition]{Remark}
\newtheorem{example}[definition]{Example}
\newtheorem{corollary}[definition]{Corollary}
\newtheorem{fact}[definition]{Fact}

\def\E{\mathbb{E}}
\def\Q{\mathbb{Q}}
\def\T{\mathbb{T}}
\def\R{\mathbb{R}}
\def\L{\mathbb{L}}
\def\Z{\mathbb{Z}}
\def\T{\mathbb{T}}
\def\RL{\mathbb{R}{\mathcal{L}}}
\def\RSpec{\mathbb{R}\underline{\mathrm{Spec}}}

\def\Tot{\mathrm{Tot}} 
\def\End{\mathrm{End}} 
\def\dr{\mathrm{dR}} 
\def\Crit{\mathrm{Crit}}
\def\Sym{\mathrm{Sym}}
\def\Out{\mathrm{Out}}
\def\Aut{\mathrm{Aut}}
\def\Tor{\mathrm{Tors}}
\def\colim{\mathrm{colim}}
\def\lim{\mathrm{lim}}
\def\Map{\mathrm{Map}}
\def\Hom{\mathrm{Hom}}
\def\holim{\mathrm{holim}}
\def\hocolim{\mathrm{hocolim}}
\def\id{\mathrm{id}}
\def\Ho{\mathrm{Ho}}
\def\Der{\mathrm{Der}}
\def\Spec{\mathrm{Spec}}
\def\Ob{\mathrm{Ob}}
\def\Mor{\mathrm{Mor}}
\def\deg{\mathrm{deg}}
\def\LocSys{\mathrm{LocSys}}
\def\pt{\mathrm{pt}}
\def\Bun{\mathrm{Bun}}

\def\Mapp{\mathbf{Map}}
\def\CAlg{\mathbf{CAlg}}
\def\Rep{\mathbf{Rep}}
\def\QCoh{\mathbf{QCoh}}
\def\Gpd{\mathbf{Gpd}}
\def\Set{\mathbf{Set}}
\def\Stk{\mathbf{Stk}}
\def\PreStk{\mathbf{PreStk}}
\def\dStk{\mathbf{dStk}}
\def\Ch{\mathbf{Ch}}
\def\P{\mathbf{P}}
\def\Pr{\mathbf{Pr}}
\def\Aff{\mathbf{Aff}}
\def\dAff{\mathbf{dAff}}
\def\Alg{\mathbf{Alg}}
\def\Mod{\mathbf{Mod}}
\def\hom{\mathbf{Hom}}
\def\RMod{\mathbf{RMod}}
\def\LMod{\mathbf{LMod}}
\def\Sch{\mathbf{Sch}}
\def\cdga{\mathbf{cdgA}}
\def\dgMod{\mathbf{dgMod}}
\def\DR{\mathbf{DR}}
\def\funl{\mathbf{Fun}^{L}}
\def\fun{\mathbf{Fun}}
\def\Moduli{\mathbf{Moduli}}

\def\A{\mathcal{A}} 
\def\C{\mathcal{C}}
\def\D{\mathcal{D}}
\def\F{\mathcal{F}}
\def\G{\mathcal{G}}
\def\K{\mathcal{K}}
\def\Z{\mathcal{Z}}

\def\X{\mathscr{X}}
\def\Dn{\mathscr{D}^{(n)}}

\def\g{\mathfrak{g}}
\def\m{\mathfrak{m}}

\def\Mfld{\mathcal{M}\mathsf{fld}}
\def\Disk{\mathcal{D}\mathsf{isk}}
\def\mfld{\mathsf{Mfld}}
\def\disk{\mathsf{Disk}}

\def\Cat{\widehat{\mathbf{Cat}}_{\infty}}
\def\Catk{\widehat{\mathbf{Cat}}_{\infty}(\mathcal{K})}

\def\aug{\text{aug}}
\def\Art{\text{Art}}

\begin{abstract}
In this paper, we prove that the naive deformation problem of an $\mathbb{E}_n$-monoidal stable $k$-linear $\infty$-category $\mathcal{C}$ is a $2$-proximate formal $\mathbb{E}_{n+2}$-moduli problem, whose corresponding formal moduli problem is controlled by the non-unital $\mathbb{E}_{n+2}$-algebra $\mathrm{fib}\big(\mathrm{End}_{\mathcal{Z}_{\mathbb{E}_n}(\mathcal{C})}(1)\rightarrow \mathrm{End}_{\mathcal{C}}(1)\big)$, where $\mathcal{Z}_{\mathbb{E}_n}(\mathcal{C})$ is the $\mathbb{E}_n$-center of $\mathcal{C}$. If $\mathcal{C}$ is rigid monoidal and tamely compactly generated by unobstructible objects, then this naive deformation problem is equivalent to the formal moduli problem. We also prove a uniqueness theorem for formal deformations of certain formal moduli problems, which can be applied to the $\mathbb{E}_1$ and $\mathbb{E}_2$-monoidal deformation problems of $\mathbf{Rep}(G)$ for a reductive algebraic group $G$ with a simple Lie algebra $\mathfrak{g}=T_e G$. Finally, we show factorization homology is compatible with deformations.
\end{abstract}

\tableofcontents

\section*{Introduction} 
\addcontentsline{toc}{section}{Introduction}
\newtheorem*{theoremA*}{Theorem A}
\newtheorem*{theoremB*}{Theorem B}
The deformation theory of $\E_n$-monoidal categories is strongly related to $n$-shifted deformation quantization in \cite{To14}, where an $n$-shifted Poisson structure on a certain derived stack $X$ can induce an $\E_n$-monoidal formal deformation of $\QCoh(X)$. This is analogous to the classical situation, where deformation quantization is related to the deformation theory of associative algebras.

\paragraph{Deformation Quantization}
Let $A$ be a commutative algebra over a field $k$ of characteristic $0$ with a Poisson bracket. An \textit{associative deformation} of $A$ over an augmented commutative algebra $B\rightarrow k$ is a (flat) right $B$-linear associative algebra $A_{B}$ such that $A_B\otimes _B k\cong A$. Let $A_\hbar=A[[\hbar]]$ be such an associative deformation over $B=k[[\hbar]]$. Then there is a bracket on $A_\hbar\otimes_{k[[\hbar]]}k=A_\hbar/(\hbar\cdot A_\hbar)\cong A$ defined by $[a,b]=\frac{a'b'-b'a'}{\hbar}\ \text{mod}\ \hbar$ where $a,\ b\in A$ and $a'$ (resp. $b'$) is a lift of $a$ (resp. $b$) in $A_\hbar$. The \textit{classical deformation quantization problem} asks for a canonical associative formal deformation $A_\hbar$ of $A$ over $k[[\hbar]]$ such that the induced bracket above is equivalent to the Poisson bracket in $A$. Therefore this problem is related to the associative deformation problem of $A$.

A concrete example is to let $A=C^{\infty}(M)$ over $k=\mathbb{R}$, consisting of $\mathbb{R}$-valued smooth functions on $M$ where $M$ is a smooth manifold. A Poisson structure on $M$ is a bilinear map $C^\infty(M)\times C^\infty(M)\rightarrow C^\infty(M)$ making $C^\infty(M)$ be a classical Poisson algebra, which is equivalent to a \textit{Poisson bivector} in $\Gamma(M,\bigwedge^2 T_M)$ where $T_M$ is the tangent bundle of $M$. With the Schouten bracket, $T_{poly}(M):=\Gamma\big(M,Sym_\mathbb{R}(T_M[-1])\big)[1]$ is a \textit{differential graded Lie algebra (dgLa)} whose differential is $0$, and a Poisson bivector $\alpha\in T^{1}_{poly}(M)=\Gamma(M,\bigwedge^2 T_M)$ will be a solution of its \textit{Maurer-Cartan equation} i.e. $d\alpha+\frac{1}{2}[\alpha,\alpha]=0$.

A philosophy tells us that every reasonable (formal) deformation problem over a field $k$ of characteristic zero should be controlled by some dgLa, which is made precise and proved by Lurie in \cite[Part IV]{SAG} using higher algebra tools. 

\begin{theorem*}[Lurie]
Let the deformation problem $\mathrm{Def}: \mathbf{Alg}_{k}^{(n),\mathrm{Art}}\rightarrow s\mathbf{Set}$ be a formal $\E_n$-moduli problem. Then 
$$\mathrm{Def}(-)\simeq \Map_{\Alg_{k}^{(n),\mathrm{aug}}}\big(\mathscr{D}^{(n)}(-),k\oplus R\big)$$
for some non-unital $\mathbb{E}_n$-algebra $R$, where $\mathscr{D}^{(n)}:(\Alg_{k}^{(n),\mathrm{aug}})^{op}\rightarrow \Alg_{k}^{(n),\mathrm{aug}}$ is the $\E_n$-Koszul duality functor. When restricted to commutative Artin algebras i.e. $\mathrm{Def}:\mathbf{CAlg}_{k}^{\mathrm{Art}}\rightarrow s\Set$, $\mathrm{Def}(-)$ will be equivalent to $\Map_{\mathbf{dgLa}_k}\big(\mathscr{D}(-),\g\big)$ where $\g\simeq R[n-1]$ and $\mathscr{D}:(\mathbf{CAlg}_{k}^{\mathrm{aug}})^{op}\rightarrow \mathbf{dgLa}_k$ is the Koszul duality functor.
\end{theorem*}

If we suppose some deformation problem is controlled by a dgLa $\g$, then for every commutative (pro-)Artin algebra $B$, the deformation set on $B$ should be 
$$\pi_0\Map_{\mathbf{dgLa}_k}(\mathscr{D}(B),\g)\simeq \mathrm{MC}(\g\otimes_k \m_B)/\sim$$
the quotient set of solutions of the Maurer-Cartan equation by gauge action. Therefore here $T_{poly}(M)$ is related to some deformation problem and a solution of its Maurer-Cartan equation i.e. a Poisson structure $\alpha$ will formally induce a formal deformation $\alpha\cdot \hbar\in \mathrm{MC}\big(T_{poly}(M)\otimes_k \hbar\cdot k[[\hbar]]\big)/\sim$ of that problem. 
\\

On the other hand, there is another dgLa $HH(A)[1]$ the shifted Hochschild cohomology complex of the algebra $A=C^{\infty}(M)$, which is related to the associative deformation problem of $A$ in some sense. In \cite[Sec. 4.6.2]{Kon03}, Kontsevich shows there is an $L_\infty$-equivalence between $T_{poly}(M)$ and $D_{poly}(M)$ where $D_{poly}(M)$ is an equivalent subalgebra of $HH(A)[1]$ consisting of \textit{polydifferential operators}. The value on $k[[\hbar]]$ of the deformation problem associated to the dgLa $D_{poly}(M)$ classifies star-products on $A$ up to gauge equivalences. A \textit{star-product} on $A$ is an associative $\mathbb{R}[[\hbar]]$-linear product on $A[[\hbar]]$ given by the formula
$$(f,g)\mapsto f\star g=fg+\hbar B_1(f,g)+\hbar^2B_2(f,g)+\cdots \in A[[\hbar]]$$
for $f,g\in A\subseteq A[[\hbar]]$.

A Poisson structure $\alpha\in T^1_{poly}(M)$ satisfying $[\alpha,\alpha]=0$ will formally induce $\alpha\cdot\hbar\in T^1_{poly}(M)[[\hbar]]$ satisfying $[\alpha\cdot \hbar,\alpha\cdot \hbar]=0$. By the $L_\infty$-equivalence above, $\alpha\cdot\hbar$ corresponds to an element in $\mathrm{MC}\big(D_{poly}(M)\otimes_k \hbar\cdot k[[\hbar]]\big)/\sim$ which is supposed to give a star-product on $A$ and the resulting associative algebra $A[[\hbar]]$ is called the \textit{deformation quantization with respect to $\alpha$}, which is compatible with the Poisson bracket on $A$.

\paragraph{Deformation Theory of Algebras}
We can study the deformation problem of associative algebras or even differential graded algebras in a more systematic way. There is another homotopically equivalent model for differential graded algebras i.e. \textit{$A_\infty$-algebras}. For a $k$-linear $A_\infty$-algebra $A$, as discussed in \cite[p13]{KoS00}, the $A_\infty$-deformation problem of $A$ is controlled by the truncated Hochschild cohomology complex
$$HH_{+}(A)=\prod_{n\geq 1}\mathrm{Hom}_{\mathbf{Mod}_k}(A^{\otimes n},A)[-n]$$
which is identified with the fiber $\mathrm{fib}(HH(A)\rightarrow A)$. It is a special case of the deformation problem of $\E_n$-algebras, because an $A_\infty$-algebra or a differential graded algebra is equivalent to an $\E_1$-algebra. Given an $\E_n$-algebra $A$ over $k$, its deformation over an Artin $\E_{n+1}$-algebra $B$ is a right $B$-linear $\E_n$-algebra $A_B$ satisfying $A_B\otimes_B k\simeq A$. In Theorem \ref{2.6}, we show the corresponding formal moduli problem of the $\E_n$-algebra $A$ is controlled by the (non-unital) $\E_{n+1}$-algebra $\mathrm{fib}(\Z_{\E_n}(A)\rightarrow A)$ where $\Z_{\E_n}(A)$ is the $\E_n$-center of $A$. If $n=1$, $\Z_{\E_1}(A)\simeq HH(A)$.

\paragraph{Deformation Theory of Categories}
From \cite[Cor. 5.1.2.6]{HA}, we know every $\E_n$-algebra $A$ corresponds to an $\E_{n-1}$-monoidal category $\LMod_A$. So the deformation problem of $\E_n$-algebras shares some relations with the deformation problem of $\E_{n-1}$-monoidal categories. In fact, in Corollary \ref{algebra} we prove their corresponding formal moduli problems are equivalent.

In \cite[Sec. 16.6]{SAG}, Lurie proves the deformation problem of a plain stable $k$-linear $\infty$-category $\C$ is a $2$-proximate formal $\E_2$-moduli problem whose corresponding formal moduli problem is equivalent to 
$$\Map_{\Alg_{k}^{(2),\aug}}(\mathscr{D}^{(2)}(-),k\oplus \mathrm{End}_{\Z_{\E_0}(\C)}(1))$$
where $\mathscr{D}^{(2)}$ is the $\E_2$-Koszul duality functor and $\mathrm{End}_{\Z_{\E_0}(\C)}(1))\simeq \End_{\End(\C)}(1)$. He proves further that if $\C$ satisfies more properties i.e. tamely compactly generated by unobstructible objects, then this deformation problem is itself a formal $\E_2$-moduli problem. We prove a similar result for $\E_n$-monoidal stable $k$-linear $\infty$-categories in Theorem \ref{mainA}, which is stated in \cite[Variation 9.20]{Lur10}. 

\begin{theoremA*}[Thm. \ref{mainA}]
Let $\C$ be a compactly generated $\E_n$-monoidal stable $k$-linear $\infty$-category. Then the deformation problem $\E_n\mathrm{CatDef}_{\C}^{c}$ consisting of compactly generated deformations of $\C$ is a $2$-proximate formal $\E_{n+2}$-moduli problem, whose corresponding formal moduli problem, denoted by $\E_n\mathrm{CatDef}_{\C}^{\wedge}$, is equivalent to
$$\Map_{\Alg_{k}^{(n+2),\mathrm{aug}}}\Big(\mathscr{D}^{(n+2)}(-),k\oplus\mathrm{fib}\big(\mathrm{End}_{\Z_{\E_n}(\C)}(1)\rightarrow \mathrm{End}_{\C}(1)\big)\Big)$$
where $\Z_{\E_n}(\C)$ is the $\E_n$-center of $\C$. If $\C$ is rigid and tamely compactly generated, then $\E_n\mathrm{CatDef}_{\C}^c$ will be a $1$-proximate formal $\E_{n+2}$-moduli problem. Moreover, if there is a collection of unobstructible objects generating $\C$ under small colimits, then $\E_n\mathrm{CatDef}_{\C}^c$ will itself be a formal $\E_{n+2}$-moduli problem.
\end{theoremA*}

\paragraph{$n$-Shifted Deformation Quantization}
We can use the theorem above to study deformation quantization in derived algebraic geometry and explain ideas in \cite{To14} in details. For some derived stack $X$, e.g. $[Y/G]$ where $Y$ is a derived affine scheme and $G$ is a reductive algebraic group over $k$, the \textit{$n$-shifted deformation quantization} $\QCoh_{\hbar}(X)$ of $\QCoh(X)$ exists, which is an $\E_n$-monoidal formal deformation of $\QCoh(X)$. 

The \textit{higher formality theorem} proved in \cite[Cor. 5.4]{To13} states that here are also two equivalent dgLas:
$$\Gamma(X,Sym_{\mathcal{O}_X}(\T_X[-n-1]))[n+1]\simeq HH^{*}_{\E_{n+1}}(X)[n+1]=\End_{\Z_{\E_n}(\QCoh(X))}(1)[n+1]$$
A solution $\beta$ of the Maurer-Cartan equation on the left hand side is called an \textit{$n$-shifted Poisson structure on $X$}. So that with the theorem above, we know the $\E_n$-monoidal formal moduli problem of $\QCoh(X)$ is controlled by the non-unital $\E_{n+2}$-algebra
$$\mathrm{fib}\big(\Gamma(X,Sym_{\mathcal{O}_X}(\T_X[-n-1]))\rightarrow \Gamma(X,\mathcal{O}_X)\big)$$
Then for $X=[Y/G]$ where $Y$ is a derived affine scheme and $G$ is a reductive group, an $n$-shifted Poisson structure $\beta$ will lie in this fiber and induce an $\E_n$-monoidal formal deformation of $\QCoh(X)$, which may be a \textit{curved deformation} different from \cite{CPTVV}, although the insight of \cite[Thm. 4.27]{BKP} suggests that such a curved phenomenon may not occur in our setting here.

A formal localization theorem is proved in \cite{CPTVV} that for a derived Artin stack $X$, $\mathbf{Perf}(X)$ is equivalent to the perfect $\mathcal{B}_X(\infty)$-modules. And an $n$-shifted Poisson structure is equivalent to a $\mathbb{P}_{n+1}$-algebra structure on $\mathcal{B}_X(\infty)$. A choice of the equivalence $\E_{n+1}\simeq \mathbb{P}_{n+1}$ will induce a formal deformation $\mathcal{B}_{X,\hbar}(\infty)$ whose module category can be regarded as an $\E_{n}$-monoidal formal deformation of $\mathbf{Perf}(X)$.
\\

Statements above can help us study a specific example $X=BG$ and prove a uniqueness theorem for formal deformations of $\Rep(G)=\QCoh(BG)$ where $G$ is a reductive group over the field $k$ of characteristic $0$ with a simple Lie algebra $\g=T_eG$ at the unit $e:\Spec\ k\rightarrow G$. In \cite{Dri90} and \cite{Dri91}, Drinfeld has already studied the monoidal and braided monoidal deformations of the \textit{abelian category} $\Rep(\g)^{\heartsuit}$ using (quasi-triangular) quasi-Hopf algebras, instead of working with monoidal categories directly. The uniqueness theorem Drinfeld proves is that non-trivial formal deformations of the enveloping algebra $U(\g)$ are unique up to isomorphisms, twistings and change of parameter, which implies non-trivial (braided) monoidal formal deformations of $\Rep(\g)^{\heartsuit}$ should also be unique. Such formal deformation is realized by $\Rep(U_\hbar(\g))^\heartsuit$ where $U_\hbar(\g)$ is the universal enveloping algebra. In particular, the Drinfeld category will be equivalent to $\Rep(U_\hbar(\g))^\heartsuit$ as braided monoidal categories \cite[Thm. 1.4.6]{BaK01}.

In this paper, we deal with the derived case which means our $\Rep(G)$ here is a \textit{stable $\infty$-category} not an abelian category, and prove a uniqueness theorem for its formal deformations. In Prop. \ref{RepG}, we compute the path components set of the space of $\E_2$-monoidal deformations of $\Rep(G)$ over $k[[\hbar]]$ where $\g$ is supposed to be simple, and we also show if two non-trivial $\E_2$-monoidal formal deformations both induce non-trivial first order deformations, then they are equivalent up to scaling i.e. $\mathrm{Aut}(k[[\hbar]])$. This uniqueness result can hold for other formal $\E_n$-moduli problems with $n\geq 2$, whose corresponding non-unital $\E_n$-algebra $R$ satisfies that $H^n(R)\cong k$ is one dimensional and $H^{n-1}(R)=H^{n+1}(R)=0$ (see Theorem \ref{unique}). These statements are summarized as follows.

\begin{theoremB*}[Thm. \ref{BG}, \ref{unique}]
Let $G$ be a reductive group over the field $k$ of characteristic $0$. Then the naive $\E_n$-monoidal deformation problem $\E_n\mathrm{CatDef}_{\Rep(G)}^c$ consisting of compactly generated deformations of $\Rep(G)$ is already a formal $\E_{n+2}$-moduli problem, which is equivalent to 
$$\Map_{\Alg_{k}^{(n+2),\mathrm{aug}}}\Big(\mathscr{D}^{(n+2)}(-),k\oplus\mathrm{fib}\big(Sym_{k}(\g[-n])^G\rightarrow k\big)\Big)$$
where $\g$ is the corresponding Lie algebra of $G$. If $\g$ is simple, then $\pi_0\E_2\mathrm{CatDef}_{\Rep(G)}^c(k[[\hbar]])\simeq \prod_{\mathbb{Z}}k$ is equipped with a natural action from $\mathrm{Aut}(k[[\hbar]])$ such that formal deformations which induce non-trivial first order deformations are unique up to this action.
\end{theoremB*}

This theoram also means that a non-trivial $2$-shifted Poisson structure $\beta\in (Sym^2 \g)^G$ on $BG$ can induce a non-trivial $\E_2$-monoidal formal deformation of $\Rep(G)$, which is \textit{uncurved} and unique. It also applies to the $\E_1$-monoidal deformation theory of $\Rep(G)$.

\paragraph{Conventions}
\begin{itemize}
    \item We always suppose our field $k$ is of characteristic $0$ and $\Mod_k$ is the $\infty$-category of $k$-modules.
    \item $s\Set$ is the $\infty$-category of (small) simplicial sets, and $\widehat{s\mathbf{Set}}$ denotes the $\infty$-category of not necessarily small simplicial sets.
    \item $\Cat$ is the $\infty$-category of not necessarily small $\infty$-categories.
\end{itemize}

\paragraph{Acknowledgments}
This paper is part of the author's master thesis at the University of Bonn. I would like to thank Pavel Safronov for introducing me the deformation quantization of $BG$, and for his patience to answer many of my questions in email. I'm also grateful to Germán Stefanich for supervising me in Bonn, from whom I learned higher algebra stuff and formal moduli problems. I'm benefited by many insightful discussions with him.

\section{Basics of Formal Moduli Problems}
Let $\Alg_{k}^{(n)}:=\mathrm{Alg}_{\E_n}(\Mod_k)$ be the $\infty$-category of $\E_n$-algebras over $k$. An \textit{augmented $\E_n$-algebra over $k$} is a morphism $A\rightarrow k$ in $\Alg_{k}^{(n)}$. So we define $\Alg_{k}^{(n),\aug}:=\Alg_{k}^{(n)}/k$ and $\CAlg_{k}^{\aug}:=\CAlg_k/k$, where $\CAlg_{k}=\mathrm{Alg}_{\E_\infty}(\Mod_k)$ consists of commutative algebras over $k$. For an augmented algebra $A$, $\m_A$ is the fiber of $A\rightarrow k$ in $\Mod_k$, which has a non-unital algebra structure, and then $A$ is actually equivalent to $k\oplus \m_A$. In fact, in \cite[Prop. 5.4.4.10]{HA} Lurie shows non-unital algebras are equivalent to augmented algebras in this way.

\begin{definition}
    \normalfont
An $\E_n$-algebra $A$ over $k$ is \textit{Artin} if it satisfies the following properties:
\begin{itemize}
    \item[(1).] $A$ is connective i.e. $H^n(A)=0$ for $n>0$.
    \item[(2).] $A$ is truncated i.e. $H^n(A)=0$ for $n\ll0$. 
    \item[(3).] Every $H^n(A)$ is a finite dimensional vector space over $k$.
    \item[(4).] If $\mathfrak{m}$ is the radical of $H^0(A)$, then the canonical map $k\rightarrow H^0(A)/\mathfrak{m}$ is an isomorphism.
\end{itemize}
\end{definition}

The subcategory of Artin $\E_n$-algebras over $k$ is denoted by $\Alg_{k}^{(n),\Art}$. Because the mapping space $\Map_{\Alg^{(n)}_k}(A,k)\simeq *$ is contractible for an Artin $\E_n$-algebra $A$, there is a fully faithful functor $\Alg_{k}^{(n),\Art}\rightarrow \Alg^{(n),\aug}_{k}$.

\begin{definition}
    \normalfont
Let $X:\Alg_{k}^{(n),\Art}\rightarrow s\Set$ be a functor. $X$ is a \textit{formal $\E_n$-moduli problem} if it satisfies the following properties:
\begin{itemize}
    \item[(1).] $X(k)\simeq *$ is contractible.
    \item[(2).] For every pullback diagram 
    \[\begin{tikzcd}
	A & {A_0} \\
	{A_{1}} & {A_{01}}
	\arrow[from=1-1, to=1-2]
	\arrow[from=1-1, to=2-1]
	\arrow["\lrcorner"{anchor=center, pos=0.125}, draw=none, from=1-1, to=2-2]
	\arrow[from=1-2, to=2-2]
	\arrow[from=2-1, to=2-2]
\end{tikzcd}\]
in $\Alg_{k}^{(n),\Art}$ such that underlying maps $H^0(A_1)\rightarrow H^0(A_{01})\leftarrow H^0(A_0)$ are surjective, the diagram 
\[\begin{tikzcd}
	X(A) & {X(A_0)} \\
	{X(A_{1})} & {X(A_{01})}
	\arrow[from=1-1, to=1-2]
	\arrow[from=1-1, to=2-1]
	\arrow["\lrcorner"{anchor=center, pos=0.125}, draw=none, from=1-1, to=2-2]
	\arrow[from=1-2, to=2-2]
	\arrow[from=2-1, to=2-2]
\end{tikzcd}\]
is also a pullback.
\end{itemize}
\end{definition}

The full subcategory of $\fun(\Alg_{k}^{(n),\Art},s\Set)$ consisting of formal $\E_n$-moduli problems is denoted by $\Moduli_{k}^{(n)}$ and $L:\fun(\Alg_{k}^{(n),\Art},s\Set)\rightarrow \Moduli_{k}^{(n)}$ is the corresponding localization functor.

In \cite[Sec. 15.2]{SAG}, there is a \textit{Koszul duality} functor $\Dn:(\Alg_{k}^{(n),\aug})^{op}\rightarrow \Alg_{k}^{(n),\aug}$ defined by the universal property that for an augmented $\E_n$-algebra $A\rightarrow k$, $\Dn(A)$ is a universal $\E_n$-algebra such that $A\otimes_k\Dn(A)$ has the augmentation extending $A$.

\begin{example}
    \normalfont
From \cite[Prop. 15.3.2.1]{SAG}, we know $k\oplus k[m]\simeq \Dn(\mathrm{Free}_{\E_n}^{\aug}(k[-n-m]))$ where $\mathrm{Free}_{\E_n}^{\aug}(k[-n-m])$ is the free $\E_n$-algebra $\mathrm{Free}_{\E_n}(k[-n-m])$ with only one variable at degree $n+m$, whose augmentation sends the variable to $0$ in $k$. Because $\mathrm{Free}_{\E_n}(k[-n-m])$ is $n$-coconnective and locally finite, from \cite[Thm. 15.2.2.1]{SAG} we have following equivalences
$$\Dn(k\oplus k[m])\simeq \Dn\circ\Dn(\mathrm{Free}_{\E_n}(k[-n-m]))\simeq \mathrm{Free}_{\E_n}(k[-n-m])$$
In particular, $\Dn(k[x]/x^2)\simeq \mathrm{Free}_{\E_n}(k[-n])$.
\end{example}

By \cite[Thm. 15.0.0.9]{SAG}, the philosophy that formal moduli problems are equivalent to \textit{differential graded Lie algebras} (dgLas), actually means there is an equivalence of $\infty$-categories $\Psi:\Alg_{k}^{(n),\aug}\rightarrow \Moduli_{k}^{(n)}$ which sends an augmented $\E_n$-algebra $A$ to the functor
$$X:B\mapsto\Map_{\Alg_{k}^{(n),\aug}}(\Dn(B),A)$$

From the example above, we get 
$$X(k\oplus k[m])\simeq \Map_{\Alg_{k}^{(n),\aug}}(\mathrm{Free}_{\E_n}(k[-n-m]),A)\simeq \m_A[n+m]$$
In particular, $X(k[x]/x^2)\simeq \m_A[n]$ is identified with $\T_{\Psi(A)}$ the \textit{tangent complex} of $\Psi_A$.\footnote{Here we should be careful. Strictly speaking, $X(k[x]/x^2)$ is a simplicial set but $\m_A[n]$ is a $k$-module or spectra. In fact, $X(k[x]/x^2)$ is equivalent to $DK(\tau^{\leq 0}(\m_A[n]))$ where $DK$ is the Dold-Kan correspondence. The tangent complex $\T_{\Psi(A)}$ is a spectra such that $\T_{\Psi(A)}(m)=X(k\oplus k[m])$, and then it is equivalent to $\m_A[n]$. The reason not to distinguish with $\m_A[n]$ and $DK(\tau^{\leq 0}(\m_A[n]))$ here is that by writing $\m_A[n]$, it is much more clear to talk about homotopy groups or cohomology groups, and their long exact sequences later.} Because $\m_A$ is the augmentation ideal of an $\E_n$-algebra $A$, $\m_A$ and equivalently $\T_{\Psi(A)}[-n]$ admit a non-unital $\E_n$-algebra structure. It is proved in \cite[Prop. 12.2.2.6]{SAG} that two formal moduli problems are equivalent if they have equivalent tangent complexes.

\paragraph{FMP vs DGLA}
In the case of $n=\infty$, the Koszul duality functor will be $\mathscr{D}:(\CAlg_{k}^{\aug})^{op}\rightarrow \mathbf{dgLa}_k$ whose left adjoint functor is the \textit{cohomological Chevalley-Eilenberg complex} \cite[Cons. 13.2.5.1]{SAG}.

It is stated in \cite[Thm. 13.0.0.2]{SAG} that there is an equivalence $\Psi:\mathbf{dgLa}_k\rightarrow \mathbf{Moduli}_k$ sending a dgLa $\g$ to the functor 
$$X=\Psi(\g):B\mapsto\Map_{\mathbf{dgLa}_k}(\mathscr{D}(B),\g)$$
and the corresponding tangent complex $\T_{\Psi(\g)}$ is identified with $\g[1]$. So this means for a formal $\E_n$-moduli problem $\Psi(A)$, $\m_{A}[n-1]$ has a dgLa structure and we should have the following commutative diagram similar to \cite[Thm. 14.3.0.1]{SAG}
\[\begin{tikzcd}
	{\Alg_{k}^{(n),\aug}} & {\mathbf{dgLa}_k} \\
	{\Moduli_{k}^{(n)}} & {\Moduli_k}
	\arrow[from=1-1, to=1-2]
	\arrow[from=1-1, to=2-1]
	\arrow[from=1-2, to=2-2]
	\arrow[from=2-1, to=2-2]
\end{tikzcd}\]
where the bottom map is defined by the composition of $\CAlg_{k}\rightarrow \Alg_{k}^{(n)}$.

In a summary, we get the following theorem.

\begin{theorem}[Lurie]
Let the deformation problem $\mathrm{Def}: \mathbf{Alg}_{k}^{(n),\mathrm{Art}}\rightarrow s\mathbf{Set}$ be a formal $\E_n$-moduli problem. Then 
$$\mathrm{Def}(-)\simeq \Map_{\Alg_{k}^{(n),\mathrm{aug}}}\big(\mathscr{D}^{(n)}(-),k\oplus R\big)$$
for some non-unital $\mathbb{E}_n$-algebra $R$, where $\mathscr{D}^{(n)}:(\Alg_{k}^{(n),\mathrm{aug}})^{op}\rightarrow \Alg_{k}^{(n),\mathrm{aug}}$ is the $\E_n$-Koszul duality functor. When restricted to commutative Artin algebras i.e. $\mathrm{Def}:\mathbf{CAlg}_{k}^{\mathrm{Art}}\rightarrow s\Set$, $\mathrm{Def}(-)$ will be equivalent to $\Map_{\mathbf{dgLa}_k}\big(\mathscr{D}(-),\g\big)$ where $\g\simeq R[n-1]$ and $\mathscr{D}:(\mathbf{CAlg}_{k}^{\mathrm{aug}})^{op}\rightarrow \mathbf{dgLa}_k$ is the Koszul duality functor.
\end{theorem}

When the deformation problem $\mathrm{Def}$ is defined on commutative Artin algebras, there is another equivalent approach to the theorem above in \cite{Pri10} using model categories and simplicial categories.

\begin{definition}
\normalfont
For any non-negative integer $n$, the \textit{algebra of polynomial differential forms} $\Omega_{n}^{\bullet}=\Omega^{\bullet}(\Delta^n)$ on the algebraic $n$-simplex is the commutative differential graded algebra (cdga) 
$$\Omega_{n}^{\bullet}:=k[t_{0},\cdots,t_n,dt_{0},\cdots,dt_{n}]/(\sum t_i-1,\sum dt_i)$$
where $t_i$'s are at degree $0$, $dt_i$'s are at degree $1$ and the product is the wedge product. 
\end{definition}

\begin{definition}
	\normalfont
For every dgLa $\g$, its corresponding deformation problem or formal moduli problem is the \textit{Maurer-Cartan functor}
$$\mathrm{MC}_\g: \mathbf{CAlg}_{k}^{\Art}\rightarrow s\mathbf{Set},\ \ (B,\m_B)\mapsto \{\mathrm{MC}(\mathfrak{g}\otimes_k\mathfrak{m}_B\otimes_k \Omega^\bullet(\Delta^n))\}_{n\geq 0}$$
\end{definition}

For a dgLa $\mathfrak{h}$, a solution $w\in \mathrm{MC}(\mathfrak{h})$ is an element in $\mathfrak{h}^1$ satisfying the \textit{Maurer-Cartan equation} $dw+\frac{1}{2}[w,w]=0$. There is a \textit{gauge group} $G(\mathfrak{h})=\mathrm{exp}(\mathfrak{h}^0)$ acting on $\mathrm{MC}(\mathfrak{h})$. The path components set $\pi_0 \mathrm{MC}_\g(B)$ can be identified with the quotient set $\mathrm{MC}(\mathfrak{g}\otimes_k\mathfrak{m}_B)/\sim$ by gauge action.

\begin{remark}
\normalfont
Because the Maurer-Cartan functor $\mathrm{MC}_\g$ is defined on $\infty$-categories, it means this concept should be well defined up to homotopy, which actually depends on the nilpotence of $\mathfrak{g}\otimes_k\mathfrak{m}_B$ i.e. the Artin property of $B$. $\mathrm{MC}_\g$ can also extend to \textit{pro-Artin algebras}. For a commutative Artin algebra $B$, the Maurer-Cartan approach gives another interpretation of $\Map_{\mathbf{dgLa}_k}\big(\mathscr{D}(B),\g\big)$.

\end{remark}

\subsection{Proximate Formal Moduli Problems}
Suppose $\C$ is an $\infty$-category with finite limits. For a morphism $f:X\rightarrow Y$ in $\C$, we have the diagonal morphism $\Delta_f:X\rightarrow X\times_{Y} X$. Repeating this construction, we obtain morphisms $\Delta^{n}_{f}:=\Delta_{\Delta^{n-1}_f}:X\rightarrow d^n(f)$. Note that $\Delta^{0}_f=f$ and $\Delta^{1}_f=\Delta_f$.

\begin{definition}
    \normalfont
A morphism $f:X\rightarrow Y$ in $\C$ is \textit{$n$-truncated} if $\Delta^{n+2}_{f}:X\rightarrow d^{n+2}(f)$ is an equivalence. An object $X$ in $\C$ is \textit{$n$-truncated} if the morphism $X\rightarrow *$ is $n$-truncated.
\end{definition}

It is clear that $(-2)$-truncated morphisms are just equivalences. In classical categories, $(-1)$-truncated morphisms are \textit{monomorphisms}, because for any two maps $u, v:Z\rightarrow X$ satisfying $f\circ u=f\circ v$, they factor through $X\times_Y X\simeq X$ which means $u=v$. So in $\infty$-categories, we also call $(-1)$-truncated morphisms \textit{monomorphisms}.

\begin{example}
\normalfont
Suppose $X$ is a Kan complex in $s\Set$. Then $X$ is an $n$-truncated simplicial set if and only if its homotopy groups $\pi_{k}(X,\mathrm{*})$ are trivial for all $k>n$. A morphism of Kan complexes is $n$-truncated if it is an isomorphism on homotopy groups for $k>n+1$ and an injection for $k=n+1$, which means its (homotopy) fiber is $n$-truncated.
\end{example}

\begin{definition}
    \normalfont
Let $X:\Alg_{k}^{(n),\Art}\rightarrow s\Set$ be a functor. $X$ is an \textit{$m$-proximate formal $\E_n$-moduli problem} for $m\geq 0$ if it satisfies the following properties:
\begin{itemize}
    \item[(1).] $X(k)\simeq *$ is contractible.
    \item[(2).] For every pullback diagram 
    \[\begin{tikzcd}
	A & {A_0} \\
	{A_{1}} & {A_{01}}
	\arrow[from=1-1, to=1-2]
	\arrow[from=1-1, to=2-1]
	\arrow["\lrcorner"{anchor=center, pos=0.125}, draw=none, from=1-1, to=2-2]
	\arrow[from=1-2, to=2-2]
	\arrow[from=2-1, to=2-2]
\end{tikzcd}\]
in $\Alg_{k}^{(n),\Art}$ such that underlying maps $H^0(A_1)\rightarrow H^0(A_{01})\leftarrow H^0(A_0)$ are surjective, the induced morphism
$$X(A)\rightarrow X(A_1)\times_{X(A_{01})}X(A_0)$$
is $(m-2)$-truncated.
\end{itemize}
\end{definition}

Clearly, $0$-proximate formal $\E_n$-moduli problems are just formal $\E_n$-moduli problems because $(-2)$-truncated morphisms are equivalences. There are some characterizations of $m$-proximate formal $\E_n$-moduli problems in \cite[Thm. 16.4.2.1]{SAG}.

\begin{theorem}
Let $X:\Alg_{k}^{(n),\mathrm{Art}}\rightarrow s\Set$ be a functor such that $X(k)\simeq *$. Then the following statements are equivalent.
\begin{itemize}
    \item[(1).] $X$ is an $m$-proximate formal $\E_n$-moduli problem.
    \item[(2).] There exists an $(m-2)$-truncated map $f:X\rightarrow Y$ where $Y$ is an $m$-proximate formal $\E_n$-moduli problem.
    \item[(3).] The unit map $X\rightarrow LX$ is $(m-2)$-truncated where $L:\fun(\Alg_{k}^{(n),\mathrm{Art}},s\Set)\rightarrow \Moduli_{k}^{(n)}$ is the localization functor. 
\end{itemize}
\end{theorem}

\subsection{Lurie's Tensor Product}
Suppose $\Cat$ is the $\infty$-category consisting of (not necessarily small) $\infty$-categories.

\begin{definition}
    \normalfont
Let $\mathcal{J}$ be a collection of simplicial sets. $\widehat{\mathbf{Cat}}_{\infty}(\mathcal{J})$ is the subcategory of $\Cat$ whose objects are $\infty$-categories admitting $J$-indexed colimits and morphisms are functors preserving $J$-indexed colimits for every $J\in \mathcal{J}$.
\end{definition}

We use $\K$ to denote the collection of all small simplicial sets. Then $\Catk$ is the $\infty$-category consisting of $\infty$-categories having all small colimits and functors preserving small colimits. Also note that $\Cat(\emptyset)=\Cat$ because it means there is no colimits taken. $\Cat$ has a natural symmetric monoidal structure where the product is just the usual Cartesian product, so that $\Cat\rightarrow N(\mathrm{Fin}_*)$ is a \textit{coCartesian fibration}.

\begin{prop}\footnote{\cite[Cor. 4.8.1.4]{HA}} The composition $\Catk\subseteq \Cat\rightarrow N(\mathrm{Fin}_*)$ is a coCartesian fibration and makes $\Catk$ be a symmetric monoidal $\infty$-category inheriting from $\Cat$ i.e. $\Catk_{\langle n \rangle}\simeq \Catk^{n}_{\langle 1 \rangle}$.
\end{prop}

The corresponding tensor product is called \textit{Lurie's tensor product}. This coCartesian fibration has the following data.

\begin{itemize}
\item Objects: for a pointed finite set $(I,*)$, a collection of $\infty$-categories 
    $$i\in I \setminus \{*\}\mapsto \C_i\in\Cat$$ 
    where each $\C_i$ has small colimits.
\item  Morphisms: for every map $\phi:(I,*)\rightarrow (J,*)$, a collection of functors 
    $$j\in J\setminus\{*\}\mapsto \big( \prod_{i\in \phi^{-1}(j)}\C_i\xrightarrow{F_j} \C_j\big)$$
    where $F_j$ preserves small colimits in each variables.
\end{itemize}

This symmetric monoidal structure is closed because we have 
\[\begin{tikzcd}
	&& \mathcal{E} \\
	{\C\times \D} & {\C\otimes \D} \\
	{[2]} & {[1]}
	\arrow[from=2-1, to=1-3]
	\arrow[from=2-1, to=2-2]
	\arrow["{\exists!}"', from=2-2, to=1-3]
	\arrow[from=3-1, to=3-2]
\end{tikzcd}\]
by the property of coCartesian fibration, and then
$$\funl(\C\otimes \D,\mathcal{E})\simeq \fun'(\C\times \D, \mathcal{E})\simeq \funl(\C,\funl(\D,\mathcal{E}))$$
where $\fun'(\C\times \D, \mathcal{E})$ consists of functors preserving small colimits in each variables. Therefore Lurie's tensor product is universal.

Let $\Pr^L\subseteq\Catk$ be the full subcategory consisting of \textit{presentable $\infty$-categories}. Then Lurie's tensor product defines a closed symmetric monoidal structure in $\Pr^L$ inheriting from $\Catk$. Details can be found in \cite[Prop. 4.8.1.15]{HA} and \cite[Prop. 5.5.3.8]{HTT}.

Similarly, if we suppose $\Pr^{St}\subseteq \Pr^L$ is the full subcategory consisting of \textit{presentable stable $\infty$-categories}, then $\Pr^{St}$ is also symmetric monoidal. Moreover, it is shown in \cite[Prop. 4.8.2.18]{HA} that $\Pr^{St}\simeq \Mod_{\mathbf{Sp}}(\Pr^L)$ where $\mathbf{Sp}\simeq \lim(\cdots\xrightarrow{\Omega}s\Set_*\xrightarrow{\Omega}s\Set_*)$ is the \textit{$\infty$-category of spectra}.

\begin{definition}
    \normalfont
A \textit{stable $k$-linear $\infty$-category} is a category in $\Pr^{St}_{k}:=\Mod_{\Mod_k}(\Pr^{St})$.
\end{definition}

Since $\Mod_{\Mod_k}(\Pr^{St})\simeq \Mod_{\Mod_k}(\Pr^{L})=:\Pr^{L}_k$ \cite[Sec. 3.4.1]{HA}, a \textit{presentable $k$-linear $\infty$-category} is stable. So here such two notions coincide.

\section{Deformations of Objects}
Let $\C$ be a stable $k$-linear $\infty$-category i.e. in $\Pr_{k}^{St}$. Suppose $A\in \Alg_{k}^{(1),\aug}$ and then there are two module categories i.e. 
$$\LMod_A(\C)=\LMod_A\otimes_{\Mod_k}\C,\ \ \text{and}\ \ \RMod_A(\C)=\C\otimes_{\Mod_k}\RMod_{A}$$
with functors
$$\LMod_A(\C)\rightarrow \LMod_k(\C)\simeq \C, \ \ \text{and}\ \ \RMod_A(\C)\rightarrow \RMod_k(\C)\simeq \C$$
There is no essential difference between left and right modules because $\LMod_{A}(\C)\simeq\RMod_{A^{\mathrm{rev}}}(\C)$ where $A^{\mathrm{rev}}$ is the algebra $A$ with opposite multiplication.

\begin{definition}
    \normalfont
Given an object $E\in \C$, a \textit{deformation of $E$ over $A$} is an object $E_A\in \RMod_A(\C)$ such that $E_A\otimes_A k\simeq E$ where $A\in \Alg_{k}^{(1),\aug}$.
\end{definition}

Let $\RMod(\C)$ be the $\infty$-category of pairs $(A,E_A)$ where $A\in \Alg_{k}^{(1)}$ and $E_A\in \RMod_{A}(\C)$. Morphisms in $\RMod(\C)$ consist of maps $A\rightarrow B$ and equivalences $E_A\otimes_A B\simeq E_B$. The forgetful functor 
$$\RMod(\C)\rightarrow \Alg_{k}^{(1)},\  (A,E_A)\mapsto A$$
is a \textit{coCartesian fibration} which means if $\RMod(\C)^{\text{coCart}}$ is a subcategory of $\RMod(\C)$ spanned by coCartesian morphisms, then we will get a \textit{left fibration} $\RMod(\C)^{\text{coCart}}\rightarrow \Alg_{k}^{(1)}$ \cite[Prop. 2.4.2.4]{HTT}. 

From \cite[Thm. 3.2.0.1]{HTT}, we roughly have an equivalence
$$\{\text{functors $\D\rightarrow \Cat$}\}/\text{homotopy}\simeq\{\text{coCartesian fibration $\C\rightarrow \D$}\}/\text{equivalence}$$
So in fact, the coCartesian fibration above just means we have a functor $\Alg_{k}^{(1)}\rightarrow \mathbf{Cat}_{\infty}$ sending $A$ to $\Mod_{A}(\C)$ and passing to the left fibration, we replace $\Mod_{A}(\C)$ by its underlying $\infty$-groupoid i.e. Kan complex, which means only equivalences exist.

Given an object $E\in \C\simeq \RMod_k(\C)$, it can be identified with the object $(k,E)$ in $\RMod(\C)^{\text{coCart}}$.

\begin{definition}
    \normalfont
For any object $E\in \C$, the \textit{$\infty$-category of deformations of $E$} is defined to be $\mathrm{Deform}[E]=\RMod(\C)^{\text{coCart}}/(k,E)$.
\end{definition}

This is related to a left fibration $\RMod(\C)^{\text{coCart}}/(k,E)\rightarrow \Alg_{k}^{(1),\aug}$, which is equivalent to the functor 
$$\mathrm{ObjDef}_E:\Alg_{k}^{(1),\aug}\rightarrow \widehat{s\Set},\ A  \mapsto (\RMod_{A}(\C)\times_{\C}\{E\})^{\simeq}$$
where $(\RMod_{A}(\C)\times_{\C}\{E\})^{\simeq}$ is the underlying $\infty$-groupoid of $\RMod_{A}(\C)\times_{\C}\{E\}$. If $A$ is Artin, then $\mathrm{ObjDef}_E(A)$ will be essentially small \cite[Cor. 16.5.2.3]{SAG}, so this defines a functor
$$\mathrm{ObjDef}_E:\Alg_{k}^{(1),\Art}\rightarrow s\Set$$
which is a $1$-proximate formal $\E_1$-moduli problem.

\begin{prop}\label{2.3}
Let $\C$ be a stable $k$-linear $\infty$-category. Given a pullback diagram 
\[\begin{tikzcd}
	A & {A_0} \\
	{A_1} & {A_{01}}
	\arrow[from=1-1, to=1-2]
	\arrow[from=1-1, to=2-1]
	\arrow["\lrcorner"{anchor=center, pos=0.125}, draw=none, from=1-1, to=2-2]
	\arrow[from=1-2, to=2-2]
	\arrow[from=2-1, to=2-2]
\end{tikzcd}\]
in $\Alg_{k}^{(1)}$, we will obtain a fully faithful functor
$$F:\LMod_{A}(\C)\rightarrow \LMod_{A_1}(\C)\times_{\LMod_{A_{01}}(\C)}\LMod_{A_{0}}(\C)$$
\end{prop}
Here left modules can also be replaced by right modules.

\begin{proof}
Consider the following diagram
\[\begin{tikzcd}
	\LMod_{A}(\C) & {\LMod_{A_0}(\C)} \\
	{\LMod_{A_1}(\C)} & {\LMod_{A_{01}}(\C)}
	\arrow[from=1-1, to=1-2]
	\arrow[from=1-1, to=2-1]
	\arrow[draw=none, from=1-1, to=2-2]
	\arrow[from=1-2, to=2-2]
	\arrow[from=2-1, to=2-2]
\end{tikzcd}\]
Every functor in the diagram above admit the right adjoint functor, so 
$$F:\LMod_{A}(\C)\rightarrow \LMod_{A_1}(\C)\times_{\LMod_{A_{01}}(\C)}\LMod_{A_{0}}(\C)$$
has the right adjoint functor $G$ which sends the pair $(M_{A_1},M_{A_0},M_{A_{01}})$ to $M_{A_1}\times_{M_{A_{01}}}M_{A_0}$ in $\LMod_{A}(\C)$. To show $F$ is fully faithful, it is enough to show 
$$M\rightarrow (A_1\otimes_A M)\times_{(A_{01}\otimes_A M)}(A_0\otimes_{A}M)$$
is an equivalence. 

Because the forgetful functor $\Alg_{k}^{(1)}\rightarrow \Mod_k$ preserves limits, $A$ is also a pullback in $\Mod_k$. But $\Mod_k$ is a stable $\infty$-category, which means pullback diagrams are also pushout diagrams. Then
\[\begin{tikzcd}
	A & {A_0} \\
	{A_1} & {A_{01}}
	\arrow[from=1-1, to=1-2]
	\arrow[from=1-1, to=2-1]
	\arrow[from=1-2, to=2-2]
	\arrow[from=2-1, to=2-2]
	\arrow["\lrcorner"{anchor=center, pos=0.125, rotate=180}, draw=none, from=2-2, to=1-1]
\end{tikzcd}\]
is a pushout diagram in $\Mod_k$, hence also in $\RMod_{A}$. Note that the tensor product functor $-\otimes_A M:\RMod_{A}\rightarrow \C$ preserves colimits and $\C$ is stable. Now we know
\[\begin{tikzcd}
	{M\simeq A\otimes _A M} & {A_0\otimes_{A}M} \\
	{A_1\otimes_A M} & {A_{01}\otimes_A M}
	\arrow[from=1-1, to=1-2]
	\arrow[from=1-1, to=2-1]
	\arrow[from=1-2, to=2-2]
	\arrow[from=2-1, to=2-2]
\end{tikzcd}\]
is a pushout diagram in $\C$, hence also a pullback diagram.
\end{proof}

The proposition above implies given a pullback diagram $A=\mathrm{pullback}(A_1\rightarrow A_{01}\leftarrow A_0)$ in $\Alg_{k}^{(1),\aug}$, the functor 
$$\mathrm{ObjDef}_E(A)\rightarrow \mathrm{ObjDef}_E(A_1)\times_{\mathrm{ObjDef}_E(A_{01})}\mathrm{ObjDef}_E(A_0)$$
is $(-1)$-truncated i.e. a homotopy equivalence onto its essential image, which means $\mathrm{ObjDef}_E$ is $1$-proximate.

In fact, combining Thm. 16.5.4.1 and Prop. 16.5.6.1 in \cite{SAG}, we have the following theorem.

\begin{theorem}
Suppose $\C$ is a stable $k$-linear $\infty$-category and $E\in \C$. Then the functor $\mathrm{ObjDef}_E:\Alg_{k}^{(1),\mathrm{Art}}\rightarrow s\Set$ is a $1$-proximate formal $\E_1$-moduli problem, whose corresponding formal $\E_1$-moduli problem $\mathrm{ObjDef}_{E}^{\wedge}:=L(\mathrm{ObjDef}_E)$ is equivalent to $\Psi(k\oplus\mathrm{End}_{\C}(E))$ where $\Psi:\Alg_{k}^{(1),\mathrm{aug}}\rightarrow \Moduli_{k}^{(1)}$. Moreover, if $\C$ admits a left complete $t$-structure and $E$ is $n$-connective, then $\mathrm{ObjDef}_E$ is itself a formal $\E_1$-moduli problem.
\end{theorem}

\begin{example}
\normalfont
Suppose $X$ is a classical scheme over $k$ and $\F$ is a coherent sheaf on $X$. Because $\F$ is connective, $\mathrm{ObjDef}_{\F}$ is already a formal $\E_1$-moduli problem. Then the first order deformations of $\F$ are classified by
\begin{align*}
    \pi_0\mathrm{ObjDef}_{\F}(k[x]/x^2)&\simeq \pi_0\Map_{\Alg_{k}^{(1),\aug}}(\mathscr{D}^{(1)}(k[x]/x^2),k\oplus\mathrm{End}_{\mathbf{Coh}(X)}(\F))\\
    &\simeq \pi_0\Map_{\Alg_{k}^{(1)}}(\mathscr{D}^{(1)}(k[x]/x^2),\mathrm{End}_{\mathbf{Coh}(X)}(\F))\\
    &\simeq \mathrm{Ext}^{1}_{\mathcal{O}_X}(\F,\F)
\end{align*}
This covers the classical result in \cite[Thm. 2.7]{Har10}.
\end{example}

\subsection{Deformations of $\E_n$-Algebras}\label{2.1}
Now we suppose $E$ is an $\E_n$-algebra i.e. an object in $\Alg_{k}^{(n)}$ and we want to study its $\E_n$-algebra deformations. In \cite[Cor. 5.1.2.6]{HA}, there is a fully faithful functor
$$\Alg_{k}^{(n+1)}\rightarrow \Alg_{\E_{n}}(\Pr_{k}^{St}),\ A\mapsto \LMod_{A}$$
The $\infty$-category $\Alg_{A}^{(n)}:=\Alg_{\E_n}(\LMod_A)$ consists of $\E_n$-algebras over $A$. The difference between $\Alg_{A}^{(n)}$ and $A/\Alg_{k}^{(n)}$ is that a map of $\E_{n}$-algebras from $A$ to $B$ does not exhibit $B$ as an $\E_n$-algebra object in $\LMod_A$ unless it factors through the $\E_n$-center $\Z_{\E_n}(A)$ \cite[Def. 5.3.1.12]{HA}. 

An \textit{$\E_n$-algebra deformation} of $E$ over $A$ is an object $E_A\in\Alg_{A}^{(n)}$ such that $k\otimes_A E_A\simeq E$. In fact, for a map $A\rightarrow B$ of $\E_{n+1}$-algebras, the functor $B\otimes_A -:\LMod_{A}\rightarrow \LMod_{B}$ is $\E_n$-monoidal \cite[Def. 2.1.3.7]{HA}, so it will induce a functor of algebras $\Alg_{A}^{(n)}\rightarrow \Alg_{B}^{(n)}$ by tensor product.

Similar to the deformation of usual objects, here we obtain a coCartesian fibration $\Alg^{(n)}\rightarrow \Alg_{k}^{(n+1)}$ where $\Alg^{(n)}$ consists of pairs $(A,E_A)$ such that $A\in \Alg_{k}^{(n+1)}$ and $E_A\in \Alg_{A}^{(n)}$. Its subcategory of coCartesian morphisms forms a left fibration $\Alg^{(n),\text{coCart}}\rightarrow \Alg_{k}^{(n+1)}$. The $\infty$-category of deformations for the $\E_n$-algebra $E$ is defined to be $\mathrm{Deform}[E]:=\Alg^{(n),\text{coCart}}/(k,E)$. The corresponding left fibration gives a functor
$$\E_n\mathrm{AlgDef}_E:\Alg_{k}^{(n+1),\aug}\rightarrow \widehat{s\mathbf{Set}},\ A\mapsto (\Alg_{A}^{(n)}\times_{\Alg_{k}^{(n)}} \{E\})^{\simeq}$$

whose value on an Artin algebra is actually an essentially small simplicial set \cite[Cor. 4.17]{BKP}.
\\

Prop. \ref{2.3} can be applied to this case with some modifications. Given a pullback diagram 
\[\begin{tikzcd}
	A & {A_0} \\
	{A_1} & {A_{01}}
	\arrow[from=1-1, to=1-2]
	\arrow[from=1-1, to=2-1]
	\arrow["\lrcorner"{anchor=center, pos=0.125}, draw=none, from=1-1, to=2-2]
	\arrow[from=1-2, to=2-2]
	\arrow[from=2-1, to=2-2]
\end{tikzcd}\]
of $\E_{n+1}$-algebras, we know 
$$\LMod_{A} \rightarrow \LMod_{A_1} \times_{\LMod_{A_{01}}}\LMod_{A_{0}} $$
is fully faithful. Because the functor $\Alg_{\E_n}:\Alg_{\E_n}(\Pr^{L})\rightarrow \Pr^{L}$ sending an $\E_n$-monoidal category $\C$ to $\Alg_{\E_n}(\C)$ consisting of $\E_n$-algebra objects, preserves limits and fully faithful functors, we get a fully faithful functor 
$$\Alg_{\E_n}(\LMod_{A})\rightarrow\Alg_{\E_n}(\LMod_{A_1})\times_{\Alg_{\E_n}(\LMod_{A_{01}})}\Alg_{\E_n}(\LMod_{A_{0}})$$
of algebras. This proves $\E_n\mathrm{AlgDef}_E$ is a $1$-proximate formal $\E_{n+1}$-moduli problem.

An $\E_n$-monoidal category $\C$ is a coCartesian  fibration $\C^\otimes\rightarrow \E_{n}^{\otimes}$ where $\C^\otimes$ is the corresponding $\infty$-operad of $\C$ and $\E_{n}^{\otimes}$ is the \textit{$\infty$-operad of little $n$-cubes}. Then $\Alg_{\E_n}(\C)$ is the full subcategory of $\fun_{/\E_{n}^{\otimes}}(\E_{n}^{\otimes},\C^\otimes)$ spanned by maps of $\infty$-operads \cite[Def. 2.1.3.1]{HA}. For a fully faithful $\E_n$-monoidal functor $\C\rightarrow \D$, the induced map on functor categories $\fun(\E_{n}^{\otimes},\C^\otimes)\rightarrow \fun(\E_{n}^{\otimes},\D^\otimes)$ will also be fully faithful. Since $\Alg_{\E_n}(-)$ is defined by conditions on functors, $\Alg_{\E_n}(\C)\rightarrow \Alg_{\E_n}(\D)$ is fully faithful as well. The property of preserving limits can be proved similarly.

As a generalization of \cite[Prop. 4.19]{BKP}, we get the following theorem.

\begin{theorem}\label{2.6}
Suppose $E$ is an $\E_n$-algebra over $k$. Then $\E_n\mathrm{AlgDef}_E$ is a $1$-proximate formal $\E_{n+1}$-moduli problem, whose corresponding formal moduli problem $\E_n\mathrm{AlgDef}_{E}^{\wedge}$ is equivalent to $\Psi\big(k\oplus \mathrm{fib}(\Z_{\E_n}(E)\rightarrow E)\big)$ where $\Psi:\Alg_{k}^{(n+1),\mathrm{aug}}\rightarrow \Moduli_{k}^{(n+1)}$ and $\Z_{\E_n}(E)$ is the $\E_n$-center of $E$. Moreover, if $E$ is $n$-connective, then $\E_n\mathrm{AlgDef}_E$ will itself be a formal $\E_{n+1}$-moduli problem.
\end{theorem}

As for the proof of the statement that $\E_n\mathrm{AlgDef}_{E}^{\wedge}$ is equivalent to $\Psi\big(k\oplus \mathrm{fib}(\Z_{\E_n}(E)\rightarrow E)\big)$, it combines Theorem \ref{mainA}, Corollary \ref{algebra} and \cite[Cor. 4.38]{Fra13}. It is proved in \cite{Fra13} that there exists a fiber sequence of non-unital $\E_{n+1}$-algebras
$$E[-1]\rightarrow\T_{\E_n\mathrm{AlgDef}_{E}^{\wedge}}[-n-1]\rightarrow \Z_{\E_n}(E)$$
where the $\E_n$-center $\Z_{\E_n}(E)$ is just the $\E_n$-Hochschild cohomology of $E$. Therefore our formal moduli problem $\E_n\mathrm{AlgDef}_{E}^{\wedge}$ when restricted to $\CAlg_{k}^{\Art}$ is equivalent to the algebraic group $\mathrm{Aut}_E$ in \cite[Def. 4.13]{Fra13}.

\paragraph{Deformations of Associative Algebras}
Now we suppose $A$ is an associative algebra i.e. an $\E_1$-algebra or equivalently $A_\infty$-algebra. Then $\E_1\mathrm{AlgDef}_A$ is a $1$-proximate formal $\E_2$-moduli problem and its corresponding formal moduli problem is controlled by the non-unital $\E_2$-algebra $\mathrm{fib}(\Z_{\E_1}(A)\rightarrow A)$ where the $\E_1$-center $\Z_{\E_1}(A)$ is just the usual Hochschild cohomology complex $HH(A)$ of $A$. This coincides with the statement in \cite[p13]{KoS00} that the $A_\infty$-deformation problem of $A$ is controlled by the truncated Hochschild complex
$$HH_{+}(A)=\prod_{n\geq 1}\mathrm{Hom}_{\mathbf{Mod}_k}(A^{\otimes n},A)[-n]$$
which is actually identified with $\mathrm{fib}(HH(A)\rightarrow A)$.

\begin{prop}
For a classical associative algebra $A$, its first order deformations are classified by $HH^2(A)$ and if $HH^3(A)=0$, then any first order deformation can extend to be a second order deformation.
\end{prop}

\begin{proof}
Since $A$ is a classical algebra especially a connective associative algebra, the naive deformation problem $\E_1\mathrm{AlgDef}_A$ is already a formal $\E_{2}$-moduli problem. 
\begin{align*}
    \E_1\mathrm{AlgDef}_A(k[x]/x^2)&\simeq\Map_{\Alg_{k}^{(2),\aug}}(\mathscr{D}^{(2)}(k[x]/x^2),k\oplus \mathrm{fib}(HH(A)\rightarrow A))\\
    &\simeq \Map_{\Alg_{k}^{(2)}}(\mathrm{Free}_{\E_2}(k[-2]),\mathrm{fib}(HH(A)\rightarrow A))\\
    &\simeq \Map_{\Mod_k}(k[-2],\mathrm{fib}(HH(A)\rightarrow A))\\
    &\simeq \mathrm{fib}\big( \Map_{\Mod_k}(k[-2],HH(A))\rightarrow  \Map_{\Mod_k}(k[-2],A) \big)
\end{align*}
Using long exact sequences, we have
$$H^1(A)=0\rightarrow\pi_0\E_1\mathrm{AlgDef}_A(k[x]/x^2)\rightarrow HH^2(A)\rightarrow H^2(A)=0$$
Therefore $\pi_0\E_1\mathrm{AlgDef}_A(k[x]/x^2)\simeq HH^2(A)$.

For the second order deformation, at first we have the following pullback diagram
\[\begin{tikzcd}
	{k[x]/x^3} & k \\
	{k[x]/x^2} & {k\oplus k[1]}
	\arrow[from=1-1, to=1-2]
	\arrow[from=1-1, to=2-1]
	\arrow["\lrcorner"{anchor=center, pos=0.125}, draw=none, from=1-1, to=2-2]
	\arrow[from=1-2, to=2-2]
	\arrow[from=2-1, to=2-2]
\end{tikzcd}\]
which means 
$$\E_1\mathrm{AlgDef}_A(k[x]/x^3)\simeq \mathrm{fib}\big(\E_1\mathrm{AlgDef}_A(k[x]/x^2)\rightarrow \E_1\mathrm{AlgDef}_A(k\oplus k[1]) \big)$$
where 
$$\E_1\mathrm{AlgDef}_A(k\oplus k[1])\simeq \Map_{\Alg_{k}^{(2)}}(\mathrm{Free}_{\E_2}(k[-3]),\mathrm{fib}(HH(A)\rightarrow A))$$
This computes $\pi_0\E_1\mathrm{AlgDef}_A(k\oplus k[1])\simeq HH^3(A)$ and we have the following exact sequence
$$\pi_0\E_1\mathrm{AlgDef}_A(k[x]/x^3)\rightarrow HH^2(A)\rightarrow HH^3(A)$$
So that if $HH^3(A)=0$, every first order deformation in $\pi_0\E_1\mathrm{AlgDef}_A(k[x]/x^2)$ has a lift of second order deformation.
\end{proof}

\paragraph{Classical Deformation Quantization}
We still suppose $A$ is a classical associative algebra. For formal deformations,
\begin{align*}
\pi_0\E_1\mathrm{AlgDef}_A(k[[\hbar]])&\simeq \pi_0\Map_{\Alg_{k}^{(2),\aug}}(\mathscr{D}^{(2)}(k[[\hbar]]),k\oplus HH_{+}(A)) \\
&\simeq \pi_0\Map_{\mathbf{dgLa}_k}(\mathscr{D}(k[[\hbar]]),HH_{+}(A)[1])\\
&\simeq \mathrm{MC}(HH_{+}(A)[1]\otimes \hbar\cdot k[[\hbar]])/\sim,\  \text{quotient by gauge action}\\
&\simeq \mathrm{MC}(HH(A)[1]\otimes \hbar\cdot k[[\hbar]])/\sim,\ \text{it only involves elements in $HH(A)$ at $\text{deg}=1, 2$}
\end{align*}

Now we assume $M$ is a smooth manifold and $A=C^{\infty}(M)$ over $k=\mathbb{R}$ consists of $\mathbb{R}$-valued smooth functions on $M$. Let $T_{poly}(M):=\Gamma\big(M,Sym_\mathbb{R}(T_M[-1])\big)[1]$ be a dgLa with $0$ differential and \textit{Schouten bracket}, where $T_M$ is the tangent bundle of $M$. Then a \textit{Poisson structure} $\alpha$ is just an element in $T^{1}_{poly}(M)=\Gamma(M,\bigwedge^2 T_M)$ satisfying the Maurer-Cartan equation i.e. $d\alpha+\frac{1}{2}[\alpha,\alpha]=0$. There is another dgLa $D_{poly}(M)$ which is an equivalent subalgebra of $HH(A)[1]$ consisting of \textit{polydifferential operators}. In \cite[Sec. 4.6.2]{Kon03}, Kontsevich shows there is an $L_\infty$-equivalence between $T_{poly}(M)$ and $D_{poly}(M)$.

A Poisson structure $\alpha\in T^1_{poly}(M)$ satisfying $[\alpha,\alpha]=0$ will formally induce $\alpha\cdot\hbar\in T^1_{poly}(M)[[\hbar]]$ satisfying $[\alpha\cdot \hbar,\alpha\cdot \hbar]=0$. By the $L_\infty$-equivalence above, $\alpha\cdot\hbar$ corresponds to an element in $\mathrm{MC}\big(HH(A)[1]\otimes_k \hbar\cdot k[[\hbar]]\big)/\sim$, which is a $k[[\hbar]]$-linear associative algebra. Such an algebra is called the \textit{deformation quantization with respect to $\alpha$}.

\section{Deformations of Plain Categories}
Suppose $\C$ is a stable $k$-linear $\infty$-category. We study its deformations as a plain $\infty$-category i.e  without any monoidal structure first.

Let $A\in \Alg_{k}^{(2),\aug}$ be an augmented $\E_2$-algebra. Then $\LMod_A$ will be an $\E_1$-monoidal category in $\Pr^{St}_{k}$. Categories in $\LMod_{\LMod_A}(\Pr^{St}_k)$ (resp. $\RMod_{\LMod_A}(\Pr^{St}_k)$) are called \textit{stable left (resp. right) $A$-linear $\infty$-categories}. For a stable $k$-linear $\infty$-category $\C$, its deformation over $A$ is a stable right $A$-linear $\infty$-category $\C_A$ such that $\C_A\otimes_A k:=\C_A\otimes_{\LMod_A}\LMod_k\simeq \C$.

Now assume $\mathbf{RCat}(k)$ is the $\infty$-category consisting of pairs $(A,\D_A)$ where $A$ is an $\E_2$-algebra and $\D_A$ is a stable right $A$-linear $\infty$-category. Morphisms in $\mathbf{RCat}(k)$ are pairs $(f,\mu)$ where $f:A\rightarrow B$ is a map of $\E_2$-algebras and $\mu:\D_A\otimes_A B\simeq \D_B$ is an equivalence. This gives a coCartesian fibration $\mathbf{RCat}(k)\rightarrow \Alg_{k}^{(2)}$ whose subcategory $\mathbf{RCat}(k)^{\text{coCart}}$ spanned by coCartesian morphisms will be a left fibration. The \textit{$\infty$-category of deformations of $\C$} is 
$$\mathrm{Deform}[\C]:=\mathbf{RCat}(k)^{\text{coCart}}/(k,\C)$$
whose corresponding left fibration over $\Alg_{k}^{(2),\aug}$ is related to a functor 
$$\mathrm{CatDef}_\C:\Alg_{k}^{(2),\aug}\rightarrow \widehat{s\mathbf{Set}},\ A\mapsto \big(\RMod_{\LMod_A}(\Pr^{St}_{k})\times_{\Pr^{St}_k}\{\C\}\big)^{\simeq}$$

It is stated in \cite[Cor. 16.6.2.2.]{SAG} that the value of $\mathrm{CatDef}_\C$ on Artin algebras is locally small, but we can work in a larger universe.

\begin{theorem}\label{3.1}
The functor $\mathrm{CatDef}_\C$ when restricted to $\Alg_{k}^{(2),\mathrm{Art}}$ defines a $2$-proximate formal $\E_2$-moduli problem in a larger universe.
\end{theorem}

\begin{proof}
We prove for every pullback diagram
\[\begin{tikzcd}
	A & {A_0} \\
	{A_1} & {A_{01}}
	\arrow[from=1-1, to=1-2]
	\arrow[from=1-1, to=2-1]
	\arrow["\lrcorner"{anchor=center, pos=0.125}, draw=none, from=1-1, to=2-2]
	\arrow[from=1-2, to=2-2]
	\arrow[from=2-1, to=2-2]
\end{tikzcd}\]
in $\Alg_{k}^{(2),\aug}$, the functor 
$$\mathrm{CatDef}_\C(A)\rightarrow \mathrm{CatDef}_\C(A_1)\times_{\mathrm{CatDef}_\C(A_{01})}\mathrm{CatDef}_\C(A_0)$$
is $0$-truncated. In fact, from Prop. \ref{2.3} we have a fully faithful functor
$$F:\LMod_{A}\rightarrow \LMod_{A_1}\times_{\LMod_{A_{01}}}\LMod_{A_{0}}$$
Then for any stable right $A$-linear $\infty$-category $\mathcal{E}\in \RMod_{\LMod_A}(\Pr^{St}_k)$, the functor 
$$\mathcal{E}\simeq\mathcal{E}\otimes_{A} A\rightarrow (\mathcal{E}\otimes_A A_1)\times_{(\mathcal{E}\otimes_{A} A_{01})}(\mathcal{E}\otimes_A A_{0})$$
is also fully faithful. That is because this functor has a right adjoint $\id_{\mathcal{E}}\otimes G$ where $G$ is the right adjoint of $F$. We have already known $\id_{\LMod_A}\simeq G\circ F$, so that $\id_{\mathcal{E}}\otimes \id_{\LMod_A}\simeq (\id_{\mathcal{E}}\otimes G)\circ (\id_{\mathcal{E}}\otimes F)$.

Given another $\D\in \RMod_{\LMod_A}(\Pr^{St}_k)$, the functor
\begin{align*}
\fun_{A}^L(\D,\mathcal{E})&\rightarrow \fun_{A}^L(\D,\mathcal{E}\otimes_A A_1)\times_{\fun_{A}^L(\D,\mathcal{E}\otimes_{A}A_{01})}\fun_{A}^{L}(\D,\mathcal{E}\otimes_{A} A_0)\\
&\simeq \fun_{A_1}^L(\D\otimes_{A} A_1,\mathcal{E}\otimes_A A_1)\times_{\fun_{A_{01}}^L(\D\otimes_{A} A_{01},\mathcal{E}\otimes_{A}A_{01})}\fun_{A_0}^{L}(\D\otimes_{A} A_0,\mathcal{E}\otimes_{A} A_0)
\end{align*}
is fully faithful. In particular, considering the subcategory consisting of equivalences, the functor \
$$\mathbf{Iso}_{A}(\D,\mathcal{E})\rightarrow \mathbf{Iso}_{A_1}(\D\otimes_{A} A_1,\mathcal{E}\otimes_A A_1)\times_{\mathbf{Iso}_{A_{01}}(\D\otimes_{A} A_{01},\mathcal{E}\otimes_{A}A_{01})}\mathbf{Iso}_{A_0}(\D\otimes_{A} A_0,\mathcal{E}\otimes_{A} A_0)$$
is fully faithful, which means
$$\mathrm{CatDef}_\C(A)\rightarrow \mathrm{CatDef}_\C(A_1)\times_{\mathrm{CatDef}_\C(A_{01})}\mathrm{CatDef}_\C(A_0)$$
is $0$-truncated. Here we can suppose $\D$ and $\mathcal{E}$ are deformations of $\C$ over $A$.
\end{proof}

Let $\mathrm{CatDef}_{\C}^\wedge$ denote the corresponding formal $\E_2$-moduli problem of $\mathrm{CatDef}_{\C}$, and then the natural morphism $\mathrm{CatDef}_\C\rightarrow \mathrm{CatDef}_{\C}^\wedge$ will be $0$-truncated. There is an example in \cite[Example 4.11]{BKP} such that $\mathrm{CatDef}_\C$ is not equivalent to $\mathrm{CatDef}_{\C}^\wedge$. In some special cases, we can find a replacement of this $2$-proximate formal $\E_2$-moduli problem which can be closer to $\mathrm{CatDef}_{\C}^\wedge$.
\\

Now we suppose our $\C$ is compactly generated and $\mathrm{CatDef}_{\C}^c$ is the subfunctor of $\mathrm{CatDef}_{\C}$ sending any augmented $\E_2$-algebra $A$ to the underlying $\infty$-groupoid of the $\infty$-category consisting of $A$-linear deformations of $\C$ which are compactly generated. In \cite[Prop. 16.6.6.2]{SAG}, Lurie shows it takes values of Artin algebras in essentially small simplicial sets. Since the inclusion functor $\mathrm{CatDef}_{\C}^c\rightarrow \mathrm{CatDef}_{\C}$ is $(-1)$-truncated, $\mathrm{CatDef}_{\C}^c$ is also a $2$-proximate formal $\E_2$-moduli problem and we can regard $\mathrm{CatDef}_{\C}^\wedge$ as the corresponding formal $\E_2$-moduli problem for both of them.

\begin{definition}
    \normalfont
Let $\C$ be a stable $k$-linear $\infty$-category. $\C$ is \textit{tamely compactly generated} if 
\begin{itemize}
    \item[(1).] $\C$ is compactly generated,
    \item[(2).] for every pair $(C,D)$ of compact objects in $\C$, $\mathrm{Ext}^{n}_{\C}(C,D):=\pi_{-n}\Map_{\C}(C,D)=0$ for $n\gg 0$.
\end{itemize}
\end{definition}

In \cite[Prop. 16.6.9.1]{SAG}, Lurie proves if $\C$ is tamely compactly generated, then the subfunctor $\mathrm{CatDef}_{\C}^{\text{tcg}}$ consisting of tamely compactly generated deformations of $\C$ is a $1$-proximate formal $\E_2$-moduli problem. In this case, we will have $\mathrm{CatDef}_{\C}^{c}=\mathrm{CatDef}_{\C}^{\text{tcg}}$, so that $\mathrm{CatDef}_{\C}^{c}$ is also $1$-proximate.

\begin{definition}
    \normalfont
Let $\C$ be a stable $k$-linear $\infty$-category. An object $C\in \C$ is \textit{unobstructible} if $C$ is compact and $\mathrm{Ext}^{n}_{\C}(C,C)=0$ for $n\geq 2$. 

We say $\C$ is \textit{tamely compactly generated by unobstructible objects} if $\C$ is tamely compactly generated and there exists a collection of unobstructible objects generating $\C$ under small colimits.
\end{definition}

It is shown in \cite[Cor. 16.6.10.3]{SAG} that if $\C$ is tamely compactly generated by unobstructible objects, then $\mathrm{CatDef}_{\C}^{c}$ is itself a formal $\E_2$-moduli problem. Later we will see a similar result also exists in $\E_n$-monoidal deformation theory. Combining all of these statements, we obtain the following theorem.

\begin{theorem}
Suppose $\C$ is a compactly generated stable $k$-linear $\infty$-category. Then $\mathrm{CatDef}_{\C}^{c}$ is a $2$-proximate formal $\E_2$-moduli problem whose corresponding formal moduli problem $\mathrm{CatDef}_{\C}^\wedge$ is equivalent to $\Psi(k\oplus \mathrm{End}_{\Z_{\E_0}(\C)}(1))$ where $\Z_{\E_0}(\C)=\fun^L(\C,\C)$ and $\Psi:\Alg_{k}^{(2),\mathrm{aug}}\rightarrow \Moduli_{k}^{(2)}$. If $\C$ is tamely compactly generated, then $\mathrm{CatDef}_{\C}^{c}$ will be $1$-proximate. Moreover, if $\C$ is tamely compactly generated by unobstructible objects, then $\mathrm{CatDef}_{\C}^{c}$ is itself a formal $\E_2$-moduli problem.
\end{theorem}

\subsection{Simultaneous Deformations}
For a stable $k$-linear $\infty$-category $\C$, its has a natural $\E_0$-structure and every object $E\in \C$ can be thought as the unit, so that the $\E_0$-monoidal deformation theory is acatually to deform the pair $(\C,E)$. Such deformations are called \textit{simultaneous deformations}.

For a pair $(\C,E)$, a deformation of it over an augmented $\E_2$-algebra $A$ is a pair $(\C_A,E_A)$ where $\C_A\in \RMod_{\LMod_A}(\Pr^{St}_k)$ satisfying $\C_A\otimes_A k\simeq \C$ and $E_A\otimes _A k\simeq E$ when identifying $\C_A\otimes_A k$ with $\C$. 

Suppose $\mathbf{RCat}_{k}^*$ is the $\infty$-category consisting of triples $(A,\D_A,E_A)$. Morphisms in it are triples $(f,u,v)$ where $f:A\rightarrow B$ is a map of $\E_2$-algebras, $u:\D_A\otimes_A B\simeq \D_B$ and $v:E_A\otimes_A B\simeq E_B$. Similarly, we get a left fibration $\mathbf{RCat}_{k}^{*,\ \text{coCart}}\rightarrow \Alg_{k}^{(2)}$. The \textit{$\infty$-category of deformations of $(\C,E)$} is defined to be $\mathbf{RCat}_{k}^{*,\ \text{coCart}}/(k,\C,E)$. This gives a functor
$$\mathrm{SimDef}_{(\C,E)}:\Alg_{k}^{(2),\aug}\rightarrow \widehat{s\mathbf{Set}},\ A\mapsto \Big(\RMod_{\LMod_{A}}((\Pr^{St}_{k})_{\Mod_k/})\times_{(\Pr^{St}_{k})_{\Mod_k/}}\{(\C,E)\}\Big)^{\simeq}$$
whose corresponding formal $\E_2$-moduli problem is denoted by $\mathrm{SimDef}_{(\C,E)}^\wedge$.

This functor is strongly related to $\mathrm{ObjDef}_E$ and $\mathrm{CatDef}_\C$ we defined previously. At first, the projection $(\Pr^{St}_{k})_{/\Mod_k} \rightarrow \Pr^{St}_{k}$ sending a pair $(\C,E)$ to $\C$ defines a morphism $\mathrm{SimDef}_{(\C,E)}\rightarrow \mathrm{CatDef}_\C$. Another morphism $\mathrm{ObjDef}_{E}^{(2)}\rightarrow \mathrm{SimDef}_{(\C,E)}$ sends an $A$-linear deformation $\E_A\in \C\otimes_k A$ to the pair $(\C\otimes_k A,E_A)$ where $\mathrm{ObjDef}_{E}^{(2)}$ is the restriction of $\mathrm{ObjDef}_{E}$ to $\Alg_{k}^{(2),\aug}$.

In \cite[Prop. 4.3]{BKP}, the authors show there are fiber sequences
\[\begin{tikzcd}
	{\mathrm{ObjDef}_{E}^{(2)}} & {\mathrm{SimDef}_{(\C,E)}} & {\mathrm{CatDef}_\C} \\
	{\mathrm{ObjDef}_{E}^{(2),\wedge}} & {\mathrm{SimDef}_{(\C,E)}^\wedge} & {\mathrm{CatDef}_{\C}^{\wedge}}
	\arrow[from=1-1, to=1-2]
	\arrow[from=1-1, to=2-1]
	\arrow[from=1-2, to=1-3]
	\arrow[from=1-2, to=2-2]
	\arrow[from=1-3, to=2-3]
	\arrow[from=2-1, to=2-2]
	\arrow[from=2-2, to=2-3]
\end{tikzcd}\]
where the top fiber sequence is in $\fun(\Alg^{(2),\Art}_k,\widehat{s\Set})$ and the bottom fiber sequence is in $\Moduli_{k}^{(2)}$. This also means we can get a fiber sequence of non-unital $\E_2$-algebras,
$$\T_{\mathrm{ObjDef}_{E}^{(2), \wedge}}[-2]=\mathrm{End}_{\C}(E)[-1]\rightarrow \T_{\mathrm{SimDef}_{(\C,E)}^{\wedge}}[-2]\rightarrow \T_{\mathrm{CatDef}_{\C}^{\wedge}}[-2]=\mathrm{End}_{\Z_{\E_0}(\C)}(1) $$
which implies 
$$\mathrm{SimDef}_{(\C,E)}^{\wedge}\simeq \Psi\big(k\oplus \mathrm{fib}(\mathrm{End}_{\Z_{\E_0}(\C)}(1)\rightarrow \mathrm{End}_{\C}(E))\big)$$
where $\Psi:\Alg_{k}^{(2),\aug}\rightarrow \Moduli_{k}^{(2)}$.
\\

So we can say $\E_0$-monoidal deformation problem for $\C\in \Pr_{k}^{St}$ is controlled by the non-unital $\E_2$-algebra $\mathrm{fib}(\mathrm{End}_{\Z_{\E_0}(\C)}(1)\rightarrow \mathrm{End}_{\C}(E))$, if we regard $E$ as the $\E_0$-unit of $\C$. As a generalization, we will see later that the $\E_n$-monoidal deformation problem of an $\E_n$-monoidal category $\C$ in $\Pr^{St}_k$ is controlled by the non-unital $\E_{n+2}$-algebra $\mathrm{fib}(\mathrm{End}_{\Z_{\E_n}(\C)}(1)\rightarrow \mathrm{End}_{\C}(1))$ i.e.
$$\E_n\mathrm{CatDef}_{\C}^\wedge\simeq \Map_{\Alg_{k}^{(n+2),\aug}}\Big(\mathscr{D}^{(n+2)}(-),k\oplus\mathrm{fib}\big(\mathrm{End}_{\Z_{\E_n}(\C)}(1)\rightarrow \mathrm{End}_{\C}(1)\big)\Big)$$
where $\Z_{\E_n}(\C)$ is the $\E_n$-center of $\C$. This statement appears in \cite[Variation 9.20]{Lur10}.

\section{Deformations of $\E_n$-Monoidal Categories}
Now we suppose $\C$ is an $\E_n$-monoidal stable $k$-linear $\infty$-category i.e. an $\infty$-category in $\Alg_{\E_n}(\Pr^{St}_k)$. Let $A\in \Alg_{k}^{(n+2),\aug}$ be an augmented $\E_{n+2}$-algebra and then $\LMod_A$ will be an $\E_{n+1}$-monoidal category in $\Pr^{St}_k$. So that we can regard $\RMod_{\LMod_A}(\Pr^{St}_k)$ as an $\E_n$-monoidal category.

A \textit{stable right $A$-linear $\E_n$-monoidal deformation} of $\C$ is an $\infty$-category $\C_A$ in $\Alg_{\E_n}(\RMod_{\LMod_A}(\Pr^{St}_k))$ such that $\C_A\otimes_A k\simeq \C$ in $\Alg_{\E_n}(\Pr^{St}_k)$. Let $\E_n\mathbf{RCat}_{k}$ be the $\infty$-category consisting of pairs $(A,\D_A)$ such that $A$ is an $\E_{n+2}$-algebra and $\D_A$ is a stable right $A$-linear $\E_n$-monoidal category. The coCartesian fibration $\E_n\mathbf{RCat}_{k}\rightarrow \Alg_{k}^{(n+2)}$ gives a left fibration $\E_n\mathbf{RCat}_{k}^{\text{coCart}}$ spanned by coCartesian morphisms. The \textit{$\infty$-category of $\E_n$-monoidal deformations of $\C$} is 
$$\mathrm{Deform}[\C]:=\E_n\mathbf{RCat}_{k}^{\text{coCart}}/(k,\C)$$
which is related to a functor
$$\E_n\mathrm{CatDef}_{\C}:\Alg_{k}^{(n+2),\aug}\rightarrow \widehat{s\mathbf{Set}},\ A\mapsto \big(\Alg_{\E_n}(\RMod_{\LMod_A}(\Pr^{St}_k))\times_{\Alg_{\E_n}(\Pr^{St}_k)}\{\C\}\big)^{\simeq}$$

\begin{prop}\label{4.1}
Let $\C$ be an $\E_n$-monoidal stable $k$-linear $\infty$-category. Then for every pullback diagram
\[\begin{tikzcd}
	A & {A_0} \\
	{A_1} & {A_{01}}
	\arrow[from=1-1, to=1-2]
	\arrow[from=1-1, to=2-1]
	\arrow["\lrcorner"{anchor=center, pos=0.125}, draw=none, from=1-1, to=2-2]
	\arrow[from=1-2, to=2-2]
	\arrow[from=2-1, to=2-2]
\end{tikzcd}\]
in $\Alg_{k}^{(n+2),\mathrm{aug}}$, the induced functor
$$\E_n\mathrm{CatDef}_{\C}(A)\rightarrow \E_n\mathrm{CatDef}_{\C}(A_1)\times_{\E_n\mathrm{CatDef}_{\C}(A_{01})}\E_n\mathrm{CatDef}_{\C}(A_0)$$
is $0$-truncated.
\end{prop}

\begin{proof}
The proof is the same as Theorem \ref{3.1}, but we need to notice $\fun^{L}_{A}=\fun_{\RMod_{\LMod_A}(\Pr^{St}_k)}$, so we should replace it by $\fun_{\Alg_{\E_n}(\RMod_{\LMod_A}(\Pr^{St}_k))}$.
\end{proof}

Therefore following \cite[Cor. 16.6.2.4]{SAG}, we know when restricted to $\Alg_{k}^{(n+2),\Art}$, $\E_n\mathrm{CatDef}_{\C}$ is a $2$-proximate formal $\E_{n+2}$-moduli problem in a larger universe. Its corresponding formal $\E_{n+2}$-moduli problem is denoted by $\E_n\mathrm{CatDef}_{\C}^{\wedge}$.
\\

If $\C$ is compactly generated (resp. tamely compactly generated), we can also define a subfunctor $\E_n\mathrm{CatDef}_{\C}^c$ (resp. $\E_n\mathrm{CatDef}_{\C}^{\text{tcg}}$) consisting of compactly generated (resp. tamely compactly generated) $\E_n$-monoidal deformations. But to make sure the monoidal structure behaves well with the compactly generated property, we should also suppose $\C$ is \textit{rigid monoidal}.

\begin{definition}
    \normalfont
A monoidal stable $\infty$-category $\C$ is \textit{rigid}, if it satisfies the following conditions:
\begin{itemize}
    \item[(1).] The unit object $1_\C$ is compact.
    \item[(2).] The right adjoint functor of the multiplication functor $\mathrm{mult}:\C\otimes \C\rightarrow \C$, denoted by $\mathrm{mult}^{R}$, preserves filtered colimits.
    \item[(3).] $\mathrm{mult}^R:\C\rightarrow \C\otimes \C$ is a functor of $\C$-bimodule categories. 
\end{itemize}
\end{definition}

\begin{prop}
Let $\C$ be a compactly generated rigid monoidal stable $\infty$-category. If $C$ and $D$ are compact objects in $\C$, then $C\otimes D$ will also be compact.
\end{prop}

\begin{proof}
Let $C$ and $D$ be compact objects in $\C$. Then given a filtered colimits $\colim A_i$ in $\C$, we have equivalences
\begin{align*}
    \Map_{\C}(C\otimes D,\colim A_i)&\simeq \Map_{\C\otimes \C}((C,D),\mathrm{mult}^R(\colim A_i))\\
    &\simeq  \Map_{\C\otimes \C}((C,D),\colim\ \mathrm{mult}^R(A_i))\\
    &\simeq \colim \Map_{\C\otimes \C}((C,D),\mathrm{mult}^R(A_i)),\ \text{since $(C,D)$ is compact in $\C\otimes \C$}\\
    &\simeq \colim \Map_\C(C\otimes D,A_i)
\end{align*}
which implies $C\otimes D$ is compact.
\end{proof}

In fact, it is shown in \cite[Lem. 9.1.5, Chap. 1]{GR19} that for a compactly generated monoidal stable $\infty$-category $\C$, it is rigid if and only if 
\begin{itemize}
    \item[(1).] $1_\C$ is compact;
    \item[(2).] the tensor product functor sends $\C^c\times \C^c$ to $\C^c$;
    \item[(3).] every compact object in $\C$ has both a left and a right dual.
\end{itemize}

\begin{example}
    \normalfont
	A derived stack $X$ is \textit{perfect} if
\begin{itemize}
    \item[(1).] the diagonal morphism $\Delta:X\rightarrow X\times X$ is affine,
    \item[(2).] the functor $\mathrm{Ind}(\mathbf{Perf}(X))\rightarrow \QCoh(X)$ is an equivalence.
\end{itemize}
For a perfect stack $X$, e.g. the quotient stack $[Y/G]$ where $Y$ is a quasi-projective derived scheme with a linear action from an affine algebraic group $G$ \cite[Cor. 3.22]{BFN10}, $\QCoh(X)$ is rigid. From \cite[Prop. 3.9]{BFN10}, we know compact and dualizable objects coincide in $\QCoh(X)$, so the unit $\mathcal{O}_X$ is compact and every compact object is dualizable.
\end{example}

\begin{theorem}
Let $A$ be an Artin $\E_n$-algebra over $k$ for $n\geq 3$. Suppose $\C$ is a stable right $A$-linear $\E_1$-monoidal category i.e. in $\Alg_{\E_1}(\RMod_{\LMod_A}(\Pr^{St}_k))$. Then $\C$ is rigid if and only if $\C\otimes_A k$ is rigid.
\end{theorem}

\begin{proof}
The idea to prove this theorem is to use the theory of \textit{universal descent morphisms} developed in \cite[Sec. D.3]{SAG} and \cite{Mat16}. Because our $A$ is an Artin algebra, considering the factorization $A\rightarrow H^0(A)\rightarrow k$ and using \cite[Prop. 3.24, 3.34, 3.35]{Mat16}, the morphism $A\rightarrow k$ is a universal descent morphism. So that a stable right $A$-linear $\infty$-category $\C$ will be equivalent to $\mathrm{Tot}(\C\otimes_{\LMod_{A}}\LMod_{k}^{\otimes(\bullet +1)})$ by \cite[Thm. D.3.5.2, Remark D.3.5.3]{SAG}. This means $\mathbf{Pr}_{A}^{St}$ is equivalent to the totalization of the cosimplicial category $\mathbf{Pr}_{k^{\otimes \bullet}}^{St}$ \cite[Theorem D.3.6.2]{SAG}. Then the rigid property of $\C$ can be reduced to that in $\mathbf{Pr}_{k \otimes_A k\otimes_A\cdots\otimes_A k}^{St}$. But the map $\mathbf{Pr}_{A}^{St}\rightarrow\mathbf{Pr}_{k \otimes_A k\otimes_A\cdots\otimes_A k}^{St}$ factors through $\mathbf{Pr}_{A}^{St}\rightarrow\mathbf{Pr}_{k}^{St}$, so it can be reduced to $\mathbf{Pr}_{k}^{St}$ i.e. $\C\otimes_A k$ as well.
\end{proof}

The theorem above means for a rigid $\E_n$-monoidal stable $k$-linear $\infty$-category $\C$, every $\E_n$-monoidal deformation of it will also be rigid.

\begin{theorem}\label{mainA}
Let $\C$ be a compactly generated $\E_n$-monoidal stable $k$-linear $\infty$-category. Then $\E_n\mathrm{CatDef}_{\C}^{c}$ is a $2$-proximate formal $\E_{n+2}$-moduli problem, whose corresponding formal moduli problem, denoted by $\E_n\mathrm{CatDef}_{\C}^{\wedge}$, is equivalent to
$$\Map_{\Alg_{k}^{(n+2),\mathrm{aug}}}\Big(\mathscr{D}^{(n+2)}(-),k\oplus\mathrm{fib}\big(\mathrm{End}_{\Z_{\E_n}(\C)}(1)\rightarrow \mathrm{End}_{\C}(1)\big)\Big)$$
where $\Z_{\E_n}(\C)$ is the $\E_n$-center of $\C$. If $\C$ is rigid and tamely compactly generated, then $\E_n\mathrm{CatDef}_{\C}^c$ will be a $1$-proximate formal $\E_{n+2}$-moduli problem. Moreover, if there is a collection of unobstructible objects generating $\C$ under small colimits, then $\E_n\mathrm{CatDef}_{\C}^c$ will itself be a formal $\E_{n+2}$-moduli problem.
\end{theorem}

\begin{proof}
The proof for that $\E_n\mathrm{CatDef}_{\C}^{c}$ sends Artin algebras to essentially small simplicial sets is the same as \cite[Prop. 16.6.6.2]{SAG} using Prop. \ref{4.1}. And then $\E_n\mathrm{CatDef}_{\C}^{c}$ is a $2$-proximate formal $\E_{n+2}$-moduli problem because $\E_n\mathrm{CatDef}_{\C}$ is $2$-proximate.

Now suppose $\C$ is rigid and tamely compactly generated. Following Lurie's proof of \cite[Prop. 16.6.9.1]{SAG}, for every Artin $\E_{n+2}$-algebra $R$ over $k$, we let $\chi(R)$ be a subcategory of $\Alg_{\E_n}(\RMod_{\LMod_R}(\Pr^{St}_k))$ consisting of rigid tamely compactly generated $\E_n$-monoidal stable right $R$-linear $\infty$-categories and compact $R$-linear functors which means they preserve compact objects. This defines a functor $\Alg_{k}^{(k+2), \Art}\rightarrow s\Set$.

Given a pullback diagram 
\[\begin{tikzcd}
	A & {A_0} \\
	{A_1} & {A_{01}}
	\arrow[from=1-1, to=1-2]
	\arrow[from=1-1, to=2-1]
	\arrow["\lrcorner"{anchor=center, pos=0.125}, draw=none, from=1-1, to=2-2]
	\arrow[from=1-2, to=2-2]
	\arrow[from=2-1, to=2-2]
\end{tikzcd}\]
in $\Alg_{k}^{(n+2),\Art}$ such that $H^0(A_1)\rightarrow H^0(A_{01})\leftarrow H^0(A_0)$ are surjective maps, the induced functor
$$F:\chi(A)\rightarrow \chi(A_1)\times_{\chi(A_{01})}\chi(A_0)$$
admits the right adjoint functor $G$ which sends the triple $(\D_{A_{1}},\D_{A_0},\simeq_{\D_{A_{01}}})$ to the full subcategory of $\D_{A_{1}}\times_{\D_{A_{01}}}\D_{A_0}$ generated by $\D_{A_{1}}^c\times_{\D_{A_{01}}^c}\D_{A_0}^c$ under small colimits and tensor products as an $\E_n$-monoidal category. Because our $\D$'s are rigid here, this subcategory will just be generated by $\D_{A_{1}}^c\times_{\D_{A_{01}}^c}\D_{A_0}^c$ under small colimits. Then \cite[Prop. 16.6.8.2]{SAG} shows $\mathrm{id}_{\chi(A)}\rightarrow G\circ F$ will be an equivalence. In our case here, it is actually an $\E_n$-monoidal equivalence which means $F$ is a fully faithful functor. In particular, 
$$\chi(A)^\simeq\rightarrow \chi(A_1)^\simeq\times_{\chi(A_{01})^\simeq}\chi(A_0)^\simeq$$
is $(-1)$-truncated and considering deformations of an $\E_n$-monoidal stable $k$-linear $\infty$-category $\C\in \chi(k)$, we conclude 
$$\E_n\mathrm{CatDef}_{\C}^{\mathrm{tcg}}(A)\rightarrow \E_n\mathrm{CatDef}_{\C}^{\mathrm{tcg}}(A_1)\times_{\E_n\mathrm{CatDef}_{\C}^{\mathrm{tcg}}(A_{01})} \E_n\mathrm{CatDef}_{\C}^{\mathrm{tcg}}(A_0)$$
is also $(-1)$-truncated.

Next from \cite[Prop. 16.6.9.2]{SAG}, in this case we have $\E_n\mathrm{CatDef}_{\C}^{\mathrm{tcg}}=\E_n\mathrm{CatDef}_{\C}^c$, so that $\E_n\mathrm{CatDef}_{\C}^c$ is a $1$-proximate formal $\E_{n+2}$-moduli problem if $\C$ is a rigid tamely compactly generated $\E_n$-monoidal stable $k$-linear $\infty$-category.

If moreover $\C$ is generated by a collection of unobstructible objects, then given a triple $(\D_{A_{1}},\D_{A_0},\simeq_{\D_{A_{01}}})$ in $\E_n\mathrm{CatDef}_{\C}^{c}(A_1)\times_{\E_n\mathrm{CatDef}_{\C}^{c}(A_{01})} \E_n\mathrm{CatDef}_{\C}^{c}(A_0)$, the full subcategory of $\D_{A_{1}}\times_{\D_{A_{01}}}\D_{A_0}$ generated by $\D_{A_{1}}^c\times_{\D_{A_{01}}^c}\D_{A_{0}}^c$ under small colimits gives a preimage of $(\D_{A_{1}},\D_{A_0},\simeq_{\D_{A_{01}}})$ in $\E_n\mathrm{CatDef}_{\C}^{c}(A)$. That is because our $\D$'s are rigid. There is a natural $\E_n$-monoidal structure in $\D_{A_{1}}^c\times_{\D_{A_{01}}^c}\D_{A_{0}}^c$ and so is the category generated by it under small colimits. Then the proof follows \cite[Thm. 16.6.10.2]{SAG} completely. The existence of such preimage implies 
$$\E_n\mathrm{CatDef}_{\C}^{c}(A)\rightarrow \E_n\mathrm{CatDef}_{\C}^{c}(A_1)\times_{\E_n\mathrm{CatDef}_{\C}^{c}(A_{01})} \E_n\mathrm{CatDef}_{\C}^{c}(A_0)$$
is an equivalence in this case.

As for the corresponding formal moduli problem $\E_n\mathrm{CatDef}_{\C}^{\wedge}$, we compute it later after discussing $\E_n$-center and $\E_n$-Hochschild cohomology.
\end{proof}

\subsection{$\E_n$-Hochschild Cohomology}
Let $\C$ be a presentable symmetric monoidal $\infty$-category whose monoidal structure distributes over colimits e.g. $\mathbf{Sp}$ and $\Pr^{L}$. For a (coherent) $\infty$-operad $\mathcal{O}$ and $\mathcal{O}$-algebra $A\in \Alg_{\mathcal{O}}(\C)$, in \cite[Sec. 2]{Fra13} and \cite[Sec. 3.3.3]{HA} the authors construct $\Mod^{\mathcal{O}}_{A}(\C)$ the \textit{$\infty$-category of $\mathcal{O}$-module objects in $\C$ over $A$}, which is naturally tensored over $\C$. We are interested in the case where $\mathcal{O}=\mathbb{E}_{n}^{\otimes}$ is the $\infty$-operad of little $n$-cubes and $A\in \Alg_{\E_n}(\C)$ is an $\E_n$-algebra in $\C$.

\begin{definition}
\normalfont
The \textit{$\mathcal{O}$-center} for an $\mathcal{O}$-algebra $A\in \Alg_{\mathcal{O}}(\C)$ is defined to be $\Z_{\mathcal{O}}(A):=\hom_{\Mod_{A}^{\mathcal{O}}(\C)}(A,A)\in \C$ where $\hom$ is the internal Hom functor taking values in $\C$. If $\mathcal{O}=\mathbb{E}_{n}^{\otimes}$ is the $\E_n$-operad, then $\Z_{\E_n}(A)$ is the \textit{$\E_n$-center} of an $\E_n$-algebra $A$ in $\C$.
\end{definition}

The $\mathcal{O}$-center $\Z_{\mathcal{O}}(A)\in \C$ admits some universal property used to define the concept of centers and centralizers in \cite[Sec. 5.3]{HA}.

\begin{remark}
    \normalfont
Our notion of $\E_n$-center is called $\E_n$-Hochschild cohomology in \cite{BFN10} and \cite{Fra13}. If $A$ is an $\E_n$-monoidal category, $\Z_{\E_n}(A)$ will be a category not a complex. But the concept $\E_n$-Hochschild cohomology is also used in \cite{To14} for a chain complex associated to a category. So our terminology `center' is to distinguish the two cases.

Note that in \cite[Thm. 5.1.3.2]{HA}, Lurie shows there is an $\E_n$-monoidal structure on $\Mod_{A}^{\E_n}(\C)$. And in \cite{Fra13}, John Francis constructs a formal $\E_{n+1}$-moduli problem controlled by $\Z_{\E_n}(A)$. In this way, he proves $\Z_{\E_n}(A)$ is an $\E_{n+1}$-algebra if $A$ is an $\E_n$-algebra, which is also proved in \cite[Cor. 5.3.1.15]{HA} using Dunn additivity theorem.
\end{remark}

\begin{definition}
\normalfont
Let $\C$ be an $\E_n$-monoidal stable $k$-linear $\infty$-category in $\Pr^{St}_k$. Its \textit{$\E_{n+1}$-Hochschild cohomology complex} is defined to be $HH_{\E_{n+1}}^*(\C):=\mathrm{End}_{\Z_{\E_n}(\C)}(1)$ which is an $\E_{n+2}$-algebra.
\end{definition}

\begin{example}
    \normalfont
Let $\C$ be a plain stable $k$-linear $\infty$-category. Then $\Z_{\E_0}(\C)=\fun^L(\C,\C)$. If $\C=\LMod_A$ for some $\E_n$-algebra $A$, $HH_{\E_{n}}^*(\C)$ will be equivalent to $\Z_{\E_{n}}(A)$ \cite[Cor. 4.38]{Fra13}. In particular, when $n=1$, $HH_{\E_1}^*(\LMod_A)$ is the usual Hochschild cohomology of $A$.

Let $\C$ be an $\E_1$-monoidal category. For $\Z_{\E_1}(\C)$, as discussed in \cite[Remark 2.4]{Fra13}, $\Mod_{\C}^{\mathcal{\E}_1}(\Pr^{St}_k)$ is equivalent to $\C$-bimodules in $\Pr^{St}_k$. So that $\Z_{\E_1}(\C)=\End_{\mathbf{BiMod}_\C(\Pr^{St}_k)}(\C)$ and it is also called the \textit{Drinfeld center of $\C$}.
\end{example}

Our $HH_{\E_{n+1}}^*(\C)$ can also be identified with the notion of \textit{$(n+2)$-fold endomorphism object} in \cite[Def. 1.0.1]{Chen25}. Following \cite[Def. 5.2.5]{Ste20}, we can define \textit{presentable $k$-linear $(\infty,n)$-categories}. Recall that $\Catk$ consists of cocomplete $\infty$-categories and colimit preserving functors.

\begin{definition}
    \normalfont
For $n=0$, let $0\Pr_{k}^{L}:=\Mod_k$. Then the \textit{$\infty$-category of presentable $k$-linear $(\infty,n)$-categories} is defined to be 
$$n\Pr^{L}_{k}:=\Mod_{\mathrm{pr},\ (n-1)\Pr^{L}_k}(\Catk)$$
where $\Mod_{\mathrm{pr}}$ means presentable modules i.e. $n\Pr^{L}_{k}$ is the full subcategory of 
$$n\Pr^{L,\wedge}_{k}:=\Mod_{(n-1)\Pr^{L}_k}(\Catk)$$
consisting of presentable categories.
\end{definition}

Hom objects in an $\infty$-category $\C\in n\Pr^{L}_n$ may not lie in $(n-1)\Pr^{L}_k$ but they are always in $(n-1)\Pr^{L,\wedge}_k$, which also holds for $\infty$-categories in $n\Pr^{L,\wedge}_k$. 

It is shown in \cite[Prop. 5.1.10]{Ste20} that there is a symmetric monoidal functor
$$\LMod:\Alg_{\E_1}(n\Pr^{L}_{k})\rightarrow (n+1)\Pr^{L}_{k},\ \C\mapsto \LMod_{\C}(n\Pr^{L}_{k})$$
so that we can get a sequence of functors
$$\LMod^n:\Alg_{\E_n}(\Mod_k)\rightarrow \Alg_{\E_{n-1}}(1\Pr^{L}_{k})\rightarrow\cdots \rightarrow \Alg_{\E_1}((n-1)\Pr^{L}_k)\rightarrow n\Pr^{L}_k$$

\begin{definition}
\normalfont
Let $\C\in n\Pr^{L}_k$ and $M\in \C$. The \textit{$n$-fold endomorphism object} is defined to be 
$$\mathrm{End}_{\C}^{n}(M):=\hom_{\mathrm{End}_{\C}^n(M)}(\id_{M}^{n-1},\id_{M}^{n-1})$$
where $\mathrm{End}_{\C}^{1}(M)=\hom_\C(M,M)$ is in $(n-1)\Pr^{L,\wedge}_k$, $\mathrm{End}_{\C}^{2}(M)=\hom_{\mathrm{End}_{\C}^{1}(M)}(\id,\id)$ is in $(n-2)\Pr^{L,\wedge}_k$ and finally $\mathrm{End}_{\C}^{n}(M)$ is in $0\Pr^{L,\wedge}_k=\Mod_k$.
\end{definition}

\begin{remark}
    \normalfont
For $\C\in n\Pr^{L}_k$, because $n\Pr^{L}_k\in (n+1)\Pr^{L}_k$, we have $\End^{n+1}_{n\Pr^{L}_k}(\C)\in \Mod_k$. In particular, for an $\E_n$-algebra $A$ over $k$, we get $\End^{n+1}(A):=\End^{n+1}_{n\Pr^{L}_k}(\LMod^{n}_A)$.
\end{remark}

\begin{prop}\label{center}
Let $A$ be an $\E_n$-algebra over $k$. Then $\End^{n+1}(A)$ is identified with $\Z_{\E_n}(A)$. If $\C$ is an $\E_n$-monoidal stable $k$-linear $\infty$-category in $\Pr^{St}_k$, then $HH^{*}_{\E_{n+1}}(\C)$ is equivalent to $\End^{n+2}(\C):=\End_{(n+1)\Pr^{L}_k}^{n+2}(\LMod_{\C}^n)$.
\end{prop}

\begin{proof}
Suppose $A\in \Alg^{(n)}_k$. From \cite[Prop. 4.36]{Fra13}, we know $\Z_{\E_{n-1}}(\LMod_A)\simeq \Mod^{\E_n}_{A}$, so that 
$$\Z_{\E_n}(A)=\End_{\Mod^{\E_n}_{A}}(1)\simeq \End_{\Z_{\E_{n-1}}(\LMod_A)}(1)$$
The same argument implies $\Z_{\E_{n-1}}(\LMod_A)\simeq \End_{\Z_{\E_{n-2}}(\LMod_{A}^2)}(1)$. Finally we have $\Z_{\E_0}(\LMod_{A}^n)=\hom_{n\Pr^{L}_k}(\LMod_{A}^n,\LMod_{A}^n)$, and hence $\Z_{\E_n}(A)\simeq \End^{n+1}(A)$.

The same arguments can be applied to $HH_{\E_{n+1}}^*(\C)=\End_{\Z_{\E_n(\C)}}(1)$.
\end{proof}

Let $\C\in n\Pr^{L}_n$ be a presentable $(\infty,n)$-category and $E\in \C$. We can also study the deformation problem deforming the object $E$. At first, there is a coCartesian fibration $\LMod(\C)\rightarrow \Alg((n-1)\Pr^{L}_k)$ where $\LMod(\C)$ consists of pairs $(\mathcal{A},E_{\mathcal{A}})$ such that $\A$ is an $\E_1$-algebra object in $(n-1)\Pr^{L}_k$ and $E_\A\in \LMod_\A(\C)$. Here we have
$$\LMod_\A(\C)\simeq \LMod_{\A}((n-1)\Pr^{L}_k)\otimes_{(n-1)\Pr^{L}_k}\C$$
Take the pullback along $\LMod^{n-1}:\Alg_{k}^{(n)}\rightarrow \Alg((n-1)\Pr^{L}_k)$,
\[\begin{tikzcd}
	{\LMod^{\mathrm{alg}}(\C)} & {\LMod(\C)} \\
	{\Alg_{k}^{(n)}} & {\Alg((n-1)\Pr^{L}_k)}
	\arrow[from=1-1, to=1-2]
	\arrow[from=1-1, to=2-1]
	\arrow["\lrcorner"{anchor=center, pos=0.125}, draw=none, from=1-1, to=2-2]
	\arrow[from=1-2, to=2-2]
	\arrow[from=2-1, to=2-2]
\end{tikzcd}\]
and we will get a new coCartesian fibration $\LMod^{\mathrm{alg}}(\C)\rightarrow \Alg_{k}^{(n)}$, whose subcategory spanned by coCartesian morphisms gives a left fibration $\LMod^{\mathrm{alg},\ \text{coCart}}(\C)\rightarrow \Alg_{k}^{(n)}$. The $\infty$-category of deformations of $E\in C$ is defined to be 
$$\mathrm{Deform}[E]:=\LMod^{\mathrm{alg},\ \text{coCart}}(\C)/(k,E)$$
Note that from $0\Pr^{L}_k=\Mod_k$, $(n-1)\Pr^{L}_k$ is actually equivalent to $\Mod_{k}^{n}$, so $\C\simeq \Mod^{n}_k\otimes_{(n-1)\Pr^{L}_k}\C$.

This defines a functor
$$\mathrm{ObjDef}_E:\Alg_{k}^{(n),\aug}\rightarrow \widehat{s\mathbf{Set}},\ A\mapsto \big(\LMod_{\LMod^{n-1}_A}(\C)\times_{\C} \{E\}\big)^{\simeq}$$
From \cite[Prop. 2.2.5, Thm. 2.4.1]{Chen25}, we obtain the following theorem.

\begin{theorem}
Let $\C\in n\Pr^{L}_{k}$ be a presentable $(\infty,n)$-category and $E\in \C$. Then $\mathrm{ObjDef}_E$ is an $n$-proximate formal $\E_{n}$-moduli problem in a larger universe, whose corresponding formal $\E_{n}$-moduli problem $\mathrm{ObjDef}_{E}^\wedge$ is equivalent to
$$\Map_{\Alg_{k}^{(n),\mathrm{aug}}}(\mathscr{D}^{(n)}(-),k\oplus \End^{n}_{\C}(E))$$
where $\mathscr{D}^{(n)}$ is the $\E_n$-Koszul duality functor.
\end{theorem}

\begin{remark}
    \normalfont
Let $\C\in  n\Pr^{L}_{k}$. The problem of deforming $\C$ as a presentable $(\infty,n)$-category is equivalent to the problem of deforming $\C$ as an object in $n\Pr^{L}_{k}$ where $n\Pr^{L}_{k}\in (n+1)\Pr^{L}_{k}$. So this defines an $(n+1)$-proximate formal $\E_{n+1}$-moduli problem $\mathrm{CatDef}_\C$ (in a larger universe), whose corresponding formal $\E_{n+1}$-moduli problem $\mathrm{CatDef}_{\C}^\wedge$ is equivalent to
$$\Map_{\Alg_{k}^{(n+1),\mathrm{aug}}}(\mathscr{D}^{(n+1)}(-),k\oplus \End^{n+1}_{n\Pr^{L}_k}(\C))$$
where $\mathscr{D}^{(n+1)}$ is the $\E_{n+1}$-Koszul duality functor.
\end{remark}

\begin{corollary}
Let $A$ be an $\E_n$-algebra over $k$. Then $\mathrm{CatDef}_{\LMod_{A}^n}^\wedge$ is a formal $\E_{n+1}$-moduli problem controlled by the augmented $\E_{n+1}$-algebra $k\oplus \Z_{\E_n}(A)$.
\end{corollary}

Since we have already known the deformation problems of objects and presentable $(\infty,n)$-categories, we can deform them together to get the simultaneous deformation problem in higher categories.

Similar to the $(\infty,1)$-category case, we have a coCartesian fibration 
$$\mathbf{LMod}((n\Pr^{L}_{k})_{\Mod_{k}^{n}/})\rightarrow \Alg((n\Pr^{L}_{k})_{\Mod_{k}^{n}/})\simeq \Alg(n\Pr^{L}_{k})$$
Now considering the following pullback,
\[\begin{tikzcd}
	{\mathbf{LCat}_{k}^{*}} & {\mathbf{LMod}((n\Pr^{L}_{k})_{\Mod_{k}^{n}/})} \\
	{\Alg_{k}^{(n+1)}} & {\Alg(n\Pr^{L}_{k})}
	\arrow[from=1-1, to=1-2]
	\arrow[from=1-1, to=2-1]
	\arrow["\lrcorner"{anchor=center, pos=0.125}, draw=none, from=1-1, to=2-2]
	\arrow[from=1-2, to=2-2]
	\arrow[from=2-1, to=2-2]
\end{tikzcd}\]
the fiber of the coCartesian fibration $\mathbf{LCat}_{k}^{*}\rightarrow \Alg_{k}^{(n+1)}$ over an $\E_{n+1}$-algebra $A$ consists of pairs $(\D_A,E_A)$ such that $\D_A\in \LMod_{\LMod_{A}^{n}}(n\Pr^{L}_k)$ and $E_A\in \D_A$. 

Let $\C\in n\Pr^{L}_n$ and $E\in \C$. The \textit{$\infty$-category of deformations of the pair $(\C,E)$} is defined to be 
$$\mathrm{Deform}[\C,E]:=\mathbf{LCat}_{k}^{*,\text{coCart}}/(k,\C,E)$$
which is a left fibration over $\Alg_{k}^{(n+1),\aug}$, so this gives a functor
$$\mathrm{SimDef}_{(\C,E)}:\Alg_{k}^{(n+1),\aug}\rightarrow \widehat{s\mathbf{Set}},\ A\mapsto\big(\LMod_{\LMod_{A}^n}((n\Pr^{L}_{k})_{\Mod_{k}^n/})\times_{(n\Pr^{L}_{k})_{\Mod_{k}^n/}}\{(\C,E)\}\big)^{\simeq}$$
There are two natural morphisms i.e. $\mathrm{SimDef}_{(\C,E)}\rightarrow \mathrm{CatDef}_{\C}$ sending the pair $(\C_A,E_A)$ to $\C_A$, and $\mathrm{ObjDef}_{E}^{(n+1)}\rightarrow \mathrm{SimDef}_{(\C,E)}$ sending $E_A$ to $(\LMod_{A}^{n}\otimes_{\Mod_{k}^n}\C, E_A)$. Here $\mathrm{ObjDef}_{E}^{(n+1)}$ means it is restricted to $\Alg_{k}^{(n+1),\aug}$.

\begin{prop}
The two morphisms above form a fiber sequence
$$\mathrm{ObjDef}_{E}^{(n+1)}\rightarrow \mathrm{SimDef}_{(\C,E)}\rightarrow \mathrm{CatDef}_{\C}$$
in $\fun(\Alg_{k}^{(n+1),\mathrm{Art}},\widehat{s\Set})$.
\end{prop}

\begin{proof}
For every augmented $\E_{n+1}$-algebra $A$, the fiber of 
$$\mathrm{SimDef}_{(\C,E)}(A)\rightarrow \mathrm{CatDef}_{\C}(A),\ (\C_A,E_A)\mapsto \C_A$$
at the trivial deformation $\LMod_{A}^{n}\otimes_{\Mod_{k}^n}\C$ is just the sub-$\infty$-groupoid of $(\LMod_{A}^{n}\otimes_{\Mod_{k}^n}\C)^{\simeq}$ consisting of deformations of $E$, which is equivalent to $\mathrm{ObjDef}_{E}^{(n+1)}(A)$.
\end{proof}

\begin{corollary}
Let $\C\in n\Pr^{L}_k$ and $E\in \C$. Then $\mathrm{SimDef}_{(\C,E)}$ is an $(n+1)$-proximate formal $\E_{n+1}$-moduli problem (in a larger universe), whose corresponding formal $\E_{n+1}$-moduli problem, denoted by  $\mathrm{SimDef}_{(\C,E)}^\wedge$, is equivalent to 
$$\Map_{\Alg_{k}^{(n+1),\mathrm{aug}}}\Big(\mathscr{D}^{(n+1)}(-),k\oplus\mathrm{fib}\big( \End^{n+1}_{n\Pr^{L}_k}(\C)\rightarrow \End_{\C}^{n}(E)\big)\Big)$$
where $\mathscr{D}^{(n+1)}$ is the $\E_{n+1}$-Koszul duality functor.
\end{corollary}

Now we can finish the proof of our Theorem \ref{mainA} following ideas from \cite[Example 3.5.1]{Chen25} and \cite{Fra13}.

\begin{proof}[Proof of Theorem \ref{mainA}]
Let $\C$ be an $\E_{n}$-monoidal presentable $k$-linear $\infty$-category in $\Pr^{L}_k$. We consider its $\E_n$-monoidal deformations with the left linear action from some $\E_{n+2}$-algebra. The deformation problem is denoted by $\E_n\mathrm{CatDef}_\C$, and $\E_n\mathrm{CatDef}_{\C}^\wedge$ is the corresponding formal $\E_{n+2}$-moduli problem.

Ar first, there is a morphism $\E_{n}\mathrm{CatDef}_{\C}\rightarrow \mathrm{CatDef}_{\LMod_{\C}^{n}}$, sending any $\E_n$-monoidal deformation of $\C$ over an Artin $\E_{n+2}$-algebra $A$ i.e. $\C_A\in \Alg_{\E_n}(\LMod^{2}_A)$, to $\LMod_{\C_{A}}^{n}$. $\LMod_{\C_{A}}^{n}$ gives a deformation of $\LMod_{\C}^n$ because the $\LMod$ functor is symmetric monoidal and then 
$$\Mod_{k}^{n+1}\otimes_{\LMod_{A}^{n+1}}\LMod_{\C_A}^n\simeq \LMod^{n}_{\Mod_k\otimes _{\LMod_A}\C_A}\simeq \LMod_{\C}^n$$

Next we can also construct a morphism $\mathrm{ObjDef}_{\LMod_{\C}^{n-1}\in \LMod_{\C}^n}^{(n+2)}\rightarrow \E_{n}\mathrm{CatDef}_{\C}$ which is given by $n$-fold endomorphism object. We know for an $\E_m$-monoidal category with $m\geq 0$, its endomorphism space of the unit should give an $\E_{m+1}$-monoidal object (see e.g. \cite[Cons. D.1.5.4]{SAG}). So this $n$-fold endomorphism object is actually an $\E_n$-monoidal $(\infty,1)$-category. 

Let $M\in A\otimes_k \LMod_{\C}^n := \LMod_{A}^{n+1}\otimes_{\Mod_{k}^{n+1}}\LMod_{\C}^n$ be a deformation of $\LMod_{\C}^{n-1}$, so that $\Mod_{k}^{n}\otimes_{\LMod_{A}^{n}}M\simeq \LMod_{\C}^{n-1}$. Then using \cite[Lem. 2.4.11]{Chen25}, we have
\begin{align*}
\Mod_k\otimes_{\LMod_A}\End^{n}_{A\otimes_k \LMod_{\C}^n}(M)&\simeq  \End^{n}_{k\otimes_A A\otimes_k \LMod_{\C}^n}(\Mod_{k}^{n}\otimes_{\LMod_{A}^{n}}M)\\
&\simeq \End^{n}_{\LMod_{\C}^n}(\LMod_{\C}^{n-1})\\
&\simeq \End_{\LMod_\C}(\C),\ \text{since $\LMod_{\C}^{n-1}$ is the unit in $\LMod_{\C}^n$}\\
&\simeq \C
\end{align*}
The key point for equivalences above is to notice $\LMod_A$ is dualizable as stated in \cite[Remark 4.8.4.8]{HA}. Hence we see $\End^{n}_{A\otimes_k \LMod_{\C}^n}(M)$ gives an $\E_n$-monoidal deformation of $\C$. 

Now we get a sequence of morphisms
\[\begin{tikzcd}
	{\mathrm{ObjDef}_{\LMod_{\C}^{n-1}\in \LMod_{\C}^n}^{(n+2)}} & {\E_{n}\mathrm{CatDef}_{\C}} & {\mathrm{CatDef}_{\LMod_{\C}^{n}}}
	\arrow["{\End^{n}}", from=1-1, to=1-2]
	\arrow["{\LMod^n}", from=1-2, to=1-3]
\end{tikzcd}\]
which is actually a fiber sequence in $\fun(\Alg_{k}^{(n+2),\Art},\widehat{s\Set})$. We explain this assertion as follows.
\\

Firstly, there is a fiber sequence of $\infty$-groupoids
\[\begin{tikzcd}
	{\mathcal{M}^{\simeq}} & {\Alg_{\E_0}(\LMod^{n}_{\mathcal{E}})^{\simeq}} & {(\LMod_{\mathcal{E}}^{n})^{\simeq}}
	\arrow[from=1-1, to=1-2]
	\arrow["{\mathrm{forgetful}}", from=1-2, to=1-3]
\end{tikzcd}\]
at $\mathcal{M}\in \LMod_{\mathcal{E}}^{n}$ where $\mathcal{E}$ is an $\E_n$-monoidal category. The fiber of $\mathcal{M}$ in $\Alg_{\E_0}(\LMod^{n}_{\mathcal{E}})^{\simeq}$ consists of the pair $(\mathcal{M},M)$. A morphism between two pairs $(\mathcal{M},M)$ and $(\mathcal{M},M')$ is actually an identity functor on $\mathcal{M}$ with an equivalence $M\simeq M'$. So this fiber should be equivalent to $\mathcal{M}^{\simeq}$. 

Next applying \cite[Cor. 5.1.2.6]{HA} iteratively, there is a fully faithful functor given by $\LMod^n$,
$$\LMod^n:\Alg_{\E_n}(\mathcal{E})\longrightarrow \Alg_{\E_0}(\LMod^{n}_{\mathcal{E}})$$
and if we suppose $\mathcal{M}$ is in the image of $\LMod^n$, then the fiber in $\Alg_{\E_n}(\mathcal{E})$ should be a sub-$\infty$-groupoid of $\mathcal{M}^{\simeq}$, which consists of those $M\in \mathcal{M}$ such that $(\mathcal{M},M)$ is equivalent to $(\LMod_{E}^{n}, \LMod_{E}^{n-1})$ for some $\E_n$-algebra $E\in \mathcal{E}$. 

In classical Morita theory, two algebras $R$ and $R'$ are Morita equivalent i.e. $\LMod_R\simeq \LMod_{R'}$, if and only if there is a compact generator $N\in \LMod_{R'}$ such that $R\simeq \End_{R'}(N)$ and the equivalence will send $\End_{R'}(N)$ to $N$ (see e.g. \cite[Thm. 4.16, 4.20]{Sch04}). A similar theorem in the $\infty$-category case is proved in \cite[Thm. 7.1.2.1]{HA}. Here it just means $E$ is given by the $n$-fold endomorphism object of some compact generator $M\in\mathcal{M}$, which should satisfy that the unit in the $i$-fold endomorphism object is also a compact generator, because here it is a Morita equivalence by $\LMod^n$. So we have the following fiber sequences of $\infty$-groupoids
\[\begin{tikzcd}
	{\mathcal{M}^{\simeq}} & {\Alg_{\E_0}(\LMod^{n}_{\mathcal{E}})^{\simeq}} & {(\LMod_{\mathcal{E}}^{n})^{\simeq}} \\
	{\mathcal{M}^{\simeq}_{cg}} & {\Alg_{\E_n}(\mathcal{E})^{\simeq}} & {(\LMod_{\mathcal{E}}^{n})^{\simeq}}
	\arrow[from=1-1, to=1-2]
	\arrow["{\mathrm{forgetful}}", from=1-2, to=1-3]
	\arrow[hook, from=2-1, to=1-1]
	\arrow["{\End^{n}}"', from=2-1, to=2-2]
	\arrow["{\LMod^n}"', hook, from=2-2, to=1-2]
	\arrow["{\LMod^n}"', from=2-2, to=2-3]
	\arrow[equal,from=2-3, to=1-3]
\end{tikzcd}\]
where $\mathcal{M}^{\simeq}_{cg}$ is the sub-$\infty$-groupoid of $\mathcal{M}^{\simeq}$ consisting of compact generators. Note that although for two objects $M$ and $M'$, their endomorphism objects may be equivalent, the inducing equivalence between $(\mathcal{M},M)$ and $(\mathcal{M},M')$ can not lies over the identity functor of $\mathcal{M}$ if $M$ is not equivalent to $M'$.
\\

Now let $A$ be an Artin $\E_{n+2}$-algebra and then $\mathcal{E}=\LMod^{2}_{A}$ is an $\E_n$-monoidal category. We consider the fiber of $\E_{n}\mathrm{CatDef}_{\C}(A)\rightarrow \mathrm{CatDef}_{\LMod_{\C}^{n}}(A)$ at the trivial deformation $\mathcal{M}=A\otimes_k \LMod_{\C}^n$. It should be the $\infty$-groupoid consisting of $\E_n$-monoidal deformations $\D$ over $A$ i.e. $\D\in \Alg_{\E_n}(\LMod_{A}^2)$ such that $k\otimes_A\D\simeq \C$ and $\LMod^{n}_\D\simeq A\otimes_k \LMod_{\C}^n$. By arguments above, it is equivalent to the sub-$\infty$-groupoid of $(A\otimes_k \LMod_{\C}^n)^{\simeq}_{cg}$ consisting of deformations of $\LMod_{\C}^{n-1}$, which is actually $\mathrm{ObjDef}_{\LMod_{\C}^{n-1}\in \LMod_{\C}^n}(A)$ because every deformation $M\in A\otimes_k \LMod_{\C}^n$ of $\LMod_{\C}^{n-1}$ is a compact generator.

Since $A$ is an Artin algebra, $A\rightarrow k$ is a universal descent morphism. Suppose $\langle M\rangle$ is the subcategory of $A\otimes_k \LMod_{\C}^{n}$ generated by $M$ under colimits. Then $\langle M \rangle$ is equivalent to $\mathrm{Tot}(k^{\otimes(\bullet +1)}\otimes_A \langle M\rangle)$. The inclusion functor $\langle M \rangle\rightarrow A\otimes_k \LMod_{\C}^{n}$ induces the morphism 
$$k^{\otimes(\bullet +1)}\otimes_A \langle M\rangle\rightarrow k^{\otimes(\bullet +1)}\otimes_A A\otimes_k \LMod_{\C}^{n}$$
which is actually an equivalence because $k\otimes_A M\simeq \LMod_{\C}^{n-1}$. This proves $\langle M \rangle$ is equivalent to $A\otimes_k \LMod_{\C}^{n}$. Hence $M$ is a compact generator in $A\otimes_k \LMod_{\C}^{n}$. By a similar argument, we can see $\id \in\End_{A\otimes_k \LMod_{\C}^{n}}^{i}(M) $ is also a compact generator, so that $\mathrm{ObjDef}_{\LMod_{\C}^{n-1}\in \LMod_{\C}^n}(A)\subseteq (A\otimes_k \LMod_{\C}^n)^{\simeq}_{cg}$.
\\

Such a fiber sequence fits into the following commutative diagram
\[\begin{tikzcd}
	{\mathrm{ObjDef}_{\LMod_{\C}^{n-1}\in \LMod_{\C}^n}^{(n+2)}} & {\E_{n}\mathrm{CatDef}_{\C}} & {\mathrm{CatDef}_{\LMod_{\C}^{n}}} \\
	{\mathrm{ObjDef}_{\LMod_{\C}^{n-1}\in \LMod_{\C}^n}^{(n+2)}} & {\mathrm{SimDef}_{(\LMod_{\C}^n,\LMod_{\C}^{n-1})}} & {\mathrm{CatDef}_{\LMod_{\C}^{n}}}
	\arrow[from=1-1, to=1-2]
	\arrow[equal, from=1-1, to=2-1]
	\arrow[from=1-2, to=1-3]
	\arrow[from=1-2, to=2-2]
	\arrow[equal, from=1-3, to=2-3]
	\arrow[from=2-1, to=2-2]
	\arrow[from=2-2, to=2-3]
\end{tikzcd}\]
where the left and right vertical morphisms are equivalences. Passing to formal ${\E_{n+2}}$-moduli problems, this implies 
$$\E_n\mathrm{CatDef}_{\C}^\wedge\simeq \mathrm{SimDef}_{(\LMod_{\C}^n,\LMod_{\C}^{n-1})}^\wedge$$

And from Prop. \ref{center}, we know $HH^{*}_{\E_{n+1}}(\C)=\mathrm{End}_{\Z_{\E_n}(\C)}(1)$ is equivalent to $\End_{(n+1)\Pr^{L}_k}^{n+2}(\LMod_{\C}^n)$. It is also clear that $\End_{\LMod_{\C}^n}^{n+1}(\LMod_{\C}^{n-1})\simeq \End_\C(1)$ because $\LMod_{\C}^{n-1}$ is the unit in $\LMod_{\C}^n$. Therefore this proves $\E_n\mathrm{CatDef}_{\C}^\wedge$ is equivalent to 
$$\Map_{\Alg_{k}^{(n+2),\mathrm{aug}}}\Big(\mathscr{D}^{(n+2)}(-),k\oplus\mathrm{fib}\big(\mathrm{End}_{\Z_{\E_n}(\C)}(1)\rightarrow \mathrm{End}_{\C}(1)\big)\Big)$$
where $\mathscr{D}^{(n+2)}$ is the $\E_{n+2}$-Koszul duality functor.
\end{proof}

\begin{corollary}\label{algebra}
Suppose $A$ is an $\E_{n+1}$-algebra. Then the formal $\E_{n+2}$-moduli problem $\E_{n+1}\mathrm{AlgDef}_{A}^\wedge$ is equivalent to $\E_{n}\mathrm{CatDef}_{\LMod_{A}}^\wedge$.
\end{corollary}

\begin{proof}
Both of them are equivalent to $\mathrm{SimDef}^{\wedge}_{(\LMod_{A}^{n+1},\LMod_{A}^{n})}$.
\end{proof}

\subsection{$n$-Shifted Deformation Quantization}
Let $X$ be a perfect stack. in \cite[Cor. 5.12]{BFN10} the authors show $\Z_{\E_n}(\QCoh(X))$ the $\E_{n}$-center of $\QCoh(X)$ is equivalent to $\QCoh(\mathcal{L}^{(n)}X)$ as an $\E_{n+1}$-monoidal category, where $\mathcal{L}^{(n)}X:=\Mapp(S^n,X)$ is the \textit{$n$-dimensional derived loop stack}. Let $j:X\rightarrow \mathcal{L}^{(n)}X$ be the constant loop morphism. Then the $\E_{n+1}$-unit in $\QCoh(\mathcal{L}^{(n)}X)$ will be $j_*(\mathcal{O}_X)$.

\begin{definition}
    \normalfont
Let $X$ be a derived stack. The \textit{$\E_{n+1}$-Hochschild cohomology} is defined to be $HH_{\E_{n+1}}^*(X):=HH_{\E_{n+1}}^*(\QCoh(X))=\End_{\Z_{\E_n}(\QCoh(X))}(1)$.
\end{definition}

So when $X$ is a perfect stack, we have an equivalence $HH_{\E_{n+1}}^*(X)\simeq \End_{\QCoh(\mathcal{L}^{(n)}X)}(j_*(\mathcal{O}_X))$. Now let $X=[Y/G]$ be a quotient stack where $Y$ is a quasi-projective derived scheme of finite presentation and $G$ is a smooth linear algebraic group. In \cite[Sec. 5]{To13}, if we replace the $\E_{n+1}$-operad by $\mathbb{P}_{n+1}$-operad, we will get 
$$HH^*_{\mathbb{P}_{n+1}}(X)=\End_{\QCoh(\mathcal{L}_{f}^{(n)}X)}(\hat{j}_*(\mathcal{O}_X))$$
where $\mathcal{L}_{f}^{(n)}X$ is the formal completion of $j:X\rightarrow \mathcal{L}^{(n)}X$ and $\hat{j}:X\rightarrow \mathcal{L}_{f}^{(n)}X$.
The HKR theorem implies 
$$HH^*_{\mathbb{P}_{n+1}}(X)\simeq \Gamma(X,Sym_{\mathcal{O}_X}(\T_X[-n-1]))$$
It is proved in \cite[Cor. 5.4]{To13} that $HH_{\mathbb{P}_{n+1}}^*(X)$ and $HH^{*}_{\E_{n+1}}(X)$ are equivalent in this case.

\begin{theorem}[Higher Formality]
Let $X=[Y/G]$ be a quotient stack such that $Y$ is a quasi-projective derived scheme of finite presentation and $G$ is a smooth linear algebraic group. Then for $n>0$, a choice of an equivalence of operad $\E_{n+1}\simeq \mathbb{P}_{n+1}$ will induce an equivalence of dgLa   
$$HH^{*}_{\E_{n+1}}(X)[n+1]\simeq HH^{*}_{\mathbb{P}_{n+1}}(X)[n+1]\simeq \Gamma(X,Sym_{\mathcal{O}_X}(\T_X[-n-1]))[n+1]$$
\end{theorem}

\begin{corollary}
Let $X=[Y/G]$ be a quotient stack such that $Y$ is a quasi-projective derived scheme of finite presentation and $G$ is a smooth linear algebraic group. Then for $n>0$, the $\E_n$-monoidal formal moduli problem $\E_{n}\mathrm{CatDef}_{\QCoh(X)}^{\wedge}$ is equivalent to 
$$\Map_{\Alg_{k}^{(n+2),\mathrm{aug}}}\Big(\mathscr{D}^{(n+2)}(-),k\oplus\mathrm{fib}\big(\Gamma(X,Sym_{\mathcal{O}_X}(\T_X[-n-1]))\rightarrow \Gamma(X,\mathcal{O}_X)\big)\Big)$$
\end{corollary}

In derived algebraic geometry, there is also a concept generalizing the usual notion of Poisson structure. 

\begin{definition}
	\normalfont
Let $X$ be a derived Artin stack locally of finite presentation. Following \cite[Sec. 3.1]{CPTVV}, \textit{the space of $n$-shifted polyvector fields of $X$} is 
$$\mathsf{Pol}(X,n)\simeq \Gamma(X,Sym_{\mathcal{O}_X}(\T_X[-n-1]))=\bigoplus_{i\geq 0}\Gamma(X,Sym^{i}_{\mathcal{O}_X}(\T_X[-n-1]))$$
The \textit{space of $n$-shifted Poisson structures} is defined to be 
$$\mathsf{Poiss}(X,n):=\Map_{\mathbf{dgLa}_{k}^{gr}}(k(2)[-1],\mathsf{Pol}(X,n)[n+1])$$
and an \textit{$n$-shifted Poisson structure} $\beta$ is an element in $\pi_0\mathsf{Poiss}(X,n)$.
\end{definition}

Here $k(2)[-1]$ is the graded dgLa with zero bracket, which is $k$ at weight $2$ and cohomological degree $1$. Then from \cite[Sec. 4.4.2]{To14}, an $n$-shifted Poisson structure $\beta$ is a family of elements $\{\beta_i\}_{i\geq 1}$ such that $\beta_i$ is an element in $\Gamma(X,Sym^{i+1}_{\mathcal{O}_X}(\T_X[-n-1])$ at cohomological degree $n+2$ satisfying 
$$d(\beta_i)+\frac{1}{2}\sum_{a+b=i}[\beta_a,\beta_b]=0$$
where $d$ is the internal differential. So that $\beta$ is determined by a solution $\sum_{i\geq 1}\beta_i\cdot\hbar^i$ of the Mauer-Cartan equation in $\Gamma(X,Sym_{\mathcal{O}_X}(\T_X[-n-1]))[n+1]\otimes \hbar\cdot k[[\hbar]]$.

\begin{example}
    \normalfont
Let $X=BG$ where $G$ is a reductive group. We know $\T_{BG}\simeq \g[1]$ where $\g$ is the tangent space at the unit $e:\Spec k\rightarrow G$. Then
$$\mathsf{Pol}(BG,n)\simeq \Gamma(BG,Sym_{\mathcal{O}_{BG}}(\g[-n]))\simeq Sym_k (\g[-n])^G$$
and 
$$\pi_0\mathsf{Poiss}(BG,n)\simeq 
\left\{\begin{array}{cc}
  (\wedge^3 \g)^G,   &n=1  \\
  (Sym^2\g)^G,   &n=2\\
  0, & n\ne 1, 2
\end{array} \right.
$$
\end{example}

\begin{theorem}[Existence]
Let $X=[Y/G]$ be a quotient stack where $Y$ is a derived affine scheme and $G$ is a reductive group. Then for $n>0$, every $n$-shifted Poisson structure on $X$ will induce an $\E_n$-monoidal formal deformation of $\QCoh(X)$ inside $\E_{n}\mathrm{CatDef}_{\QCoh(X)}^{\wedge}(k[[x]])$.
\end{theorem}

\begin{proof}
Because here our $Y$ is a derived affine scheme and $G$ is reductive, the $G$-equivariant global section $\Gamma(X,\mathcal{O}_X)=\Gamma(Y,\mathcal{O}_Y)^G$ will take values in non-positively graded cdgas. This implies 
\begin{align*}
\pi_0\E_{n}\mathrm{CatDef}_{\QCoh(X)}^{\wedge}(k[[x]])&\simeq \pi_0\Map_{\Alg_{k}^{(n+2),\mathrm{aug}}}\Big(\mathscr{D}^{(n+2)}(k[[x]]),k\oplus\mathrm{fib}\big(HH^*_{\E_{n+1}}(X)\rightarrow \Gamma(X,\mathcal{O}_X)\big)\Big)\\
&\simeq \pi_0\Map_{\Alg_{k}^{(n+2),\aug}}\big(\mathscr{D}^{(n+2)}(k[[x]]),k\oplus HH^*_{\E_{n+1}}(X)\big),\ \text{by Lemma \ref{compare}}\\
&\simeq \pi_0\Map_{\mathbf{dgLa}_k}(\mathscr{D}(k[[x]]), HH^{*}_{\E_{n+1}}(X)[n+1])\\
&\simeq \pi_0\Map_{\mathbf{dgLa}_k}(\mathscr{D}(k[[x]]), \Gamma(X,Sym_{\mathcal{O}_X}(\T_X[-n-1]))[n+1])\\
&\simeq \mathrm{MC}\big( \Gamma(X,Sym_{\mathcal{O}_X}(\T_X[-n-1]))[n+1]\otimes x\cdot k[[x]]\big)/\sim
\end{align*}
\end{proof}

Because for an $\E_n$-monoidal stable $k$-linear $\infty$-category $\C$, the naive deformation problem $\E_n\mathrm{CatDef}_{\C}^c$ is only $2$-proximate, it seems the $n$-shifted deformation quantization may give a \textit{curved deformation}, which does not lie in the naive deformation problem. But in some good case, this problem does not occur.

\subsection{Applications to $BG$}
Let $G$ be a reductive group. We know $\Rep(G)\simeq \QCoh(BG)$ and $\T_{BG}\simeq \g[1]$ where $\g$ is the Lie algebra of the tangent space of $G$ at the unit $e:\Spec k\rightarrow G$. Clearly, $BG=[*/G]$ satisfies assumptions of the higher formality theorem. Therefore for $n>0$,
$$HH^{*}_{\E_{n+1}}(BG)\simeq \Gamma(BG,Sym_{\mathcal{O}_{BG}}(\g[-n]))\simeq Sym_{k}(\g[-n])^G$$
is related to the deformation problem of $\Rep(G)$. 

\begin{theorem}\label{BG}
Let $G$ be a reductive group. Then $\Rep(G)$ is rigid and tamely compactly generated by unobstructible objects. So that the naive deformation problem $\E_n\mathrm{CatDef}_{\Rep(G)}^c$ is already a formal $\E_{n+2}$-moduli problem, which is equivalent to 
$$\Map_{\Alg_{k}^{(n+2),\mathrm{aug}}}\Big(\mathscr{D}^{(n+2)}(-),k\oplus\mathrm{fib}\big(Sym_{k}(\g[-n])^G\rightarrow k\big)\Big)$$
where $\mathscr{D}^{(n+2)}$ is the $\E_{n+2}$-Koszul duality functor.
\end{theorem}

\begin{proof}
Because $BG=[*/G]$ is perfect, $\Rep(G)\simeq \QCoh(BG)$ is rigid and compact objects are perfect complexes i.e. bounded complexes of finite dimensional $G$-representations.

From \cite[Lem. 2.4.1]{DG13}, we know the $G$-invariant functor i.e. global section functor 
$$\Gamma(BG,-):\QCoh(BG)\rightarrow \Mod_k$$
is exact because in the field $k$ of characteristic $0$, $\QCoh(BG)^\heartsuit$ is semi-simple. Then for two perfect complexes $V$ and $W$, the Hom complex
$$\mathrm{Hom}_{\Rep(G)}(V,W)\simeq \Hom_{\Mod_k}(V,W)^G$$
is bounded on the right. This proves $\Rep(G)$ is tamely compactly generated.

Next, objects of the form $V[n]$ where $V$ is a finite dimensional $G$-representation, generate $\Rep(G)$ under colimits. But
$$\Hom_{\Rep(G)}(V[n],V[n])=\Hom_{\Mod_k}(V,V)^G$$
is concentrated at degree $0$, so $\Rep(G)$ is generated by unobstructible objects.
\end{proof}

\paragraph{Davydov-Yetter Cohomology}
In \cite[Sec. 7]{ENO05}, there is a concept of \textit{Davydov-Yetter cohomology} used to study the monoidal deformation of a tensor category.

\begin{definition}
    \normalfont
Let $\C$ be a tensor category. The \textit{Davydov-Yetter (co)chain complex $DY(\C)$} is defined as follows.
\begin{itemize}
    \item[(1).] $DY^n(\C)=\End(T_n)$ where $T_n$ is the $n$-functor 
    $$T_n:\C^n\rightarrow \C,\ (X_1,\cdots,X_n)\mapsto X_1\otimes \cdots \otimes X_n$$
    In particular, $T_0:\C^0\rightarrow \C$ is defined by $T_0(\emptyset)=1_\C$ and $T_1=\id$.
    \item[(2).] The differential $d:DY^n(\C)\rightarrow DY^{n+1}(\C)$ is given by 
    $$df=\id\otimes f_{2,\cdots,n+1}-f_{12,3,\cdots,n+1}+f_{1,23,\cdots,n+1}-\cdots+(-1)^nf_{1,\cdots,nn+1}+(-1)^{n+1}f_{1,\cdots,n}\otimes \id$$
\end{itemize}
$H^n(DY(\C))$ is called the \textit{$n$-th Davydov-Yetter cohomology}.
\end{definition}

\begin{prop}
Let $\C$ be a tensor category. The first order monoidal deformations and the obstructions to extend to be a second order deformation of $\C$ are classified by $H^3(DY(\C))$ and $H^4(DY(\C))$ respectively.
\end{prop}

\begin{example}
    \normalfont
Let the tensor category be $\C=\Rep(G)$ where $G$ is a reductive group over $k$ with the corresponding Lie algebra $\g$. \cite[Example 7.3]{ENO05 } computes $H^n(DY(\Rep(G)))=(\wedge^n \g)^G$. So that $DY(\Rep(G))$ is actually equivalent to $Sym_k(\g[-1])^G$. The space of the first order monoidal deformations should be 
\begin{align*}
\pi_0\E_1\mathrm{CatDef}_{\Rep(G)}^c(k[x]/x^2)&\simeq \pi_0\Map_{\Alg_{k}^{(3),\mathrm{aug}}}\Big(\mathscr{D}^{(3)}(k[x]/x^2),k\oplus\mathrm{fib}\big(Sym_{k}(\g[-1])^G\rightarrow k\big)\Big)\\
&\simeq \pi_0\Map_{\Alg_{k}^{(3)}}\Big(\mathrm{Free}_{\E_3}k[-3],\mathrm{fib}\big(Sym_{k}(\g[-1])^G\rightarrow k\big)\Big)\\
&\simeq \pi_0\Map_{\Mod_{k}}\Big(k[-3], \mathrm{fib}\big(Sym_{k}(\g[-1])^G\rightarrow k\big)\Big)\\
&\simeq \pi_0\mathrm{fib}\Big(\Map_{\Mod_{k}}(k[-3],Sym_k(\g[-1])^G)\rightarrow \Map_{\Mod_{k}}(k[-3],k)\Big)\\
&\simeq (\wedge^3 \g )^G
\end{align*}
which coincides with the Davydov-Yetter theory. For the second order monoidal deformation, using the pullback diagram
\[\begin{tikzcd}
	{k[x]/x^3} & k \\
	{k[x]/x^2} & {k\oplus k[1]}
	\arrow[from=1-1, to=1-2]
	\arrow[from=1-1, to=2-1]
	\arrow["\lrcorner"{anchor=center, pos=0.125}, draw=none, from=1-1, to=2-2]
	\arrow[from=1-2, to=2-2]
	\arrow[from=2-1, to=2-2]
\end{tikzcd}\]
we can obtain an exact sequence
$$\pi_0\E_1\mathrm{CatDef}_{\Rep(G)}^c(k[x]/x^3)\rightarrow \pi_0\E_1\mathrm{CatDef}_{\Rep(G)}^c(k[x]/x^2)=(\wedge^3 \g)^G\rightarrow (\wedge^4\g)^G$$

If $\g$ is simple, we see $(\wedge^3 \g)^G=H^3(\g,k)\cong k$ and $(\wedge^4 \g)^G=0$. So that in this case there is a unique non-trivial first order monoidal deformation of $\Rep(G)$ up to scaling which can extend to be a second order monoidal deformation.
\end{example}

\subsubsection{Uniqueness of the Formal Deformation}
We suppose $G$ is a reductive group and the corresponding Lie algebra $\g$ is simple. In \cite{Dri90} and \cite{Dri91}, Drinfeld has already studied the monoidal and braided monoidal deformations of the \textit{abelian category} $\Rep(\g)^{\heartsuit}$ using (quasi-triangular) quasi-Hopf algebras, instead of working with monoidal categories directly. The uniqueness theorem Drinfeld proves states that the non-trivial formal deformation of the enveloping algebra $U(\g)$ is unique up to isomorphisms, twistings and change of parameter, which implies that the non-trivial (braided) monoidal formal deformation of $\Rep(\g)^{\heartsuit}$ should also be unique. Such formal deformation is realized by $\Rep(U_\hbar(\g))^\heartsuit$ where $U_\hbar(\g)$ is the universal enveloping algebra. In particular, the Drinfeld category will be equivalent to $\Rep(U_\hbar(\g))^\heartsuit$ as braided monoidal categories \cite[Thm. 1.4.6]{BaK01}.

In the following, we work in the derived case and prove a uniqueness theorem for $\E_1$ and $\E_2$-monoidal formal deformations of the stable $\infty$-category $\Rep(G)$.

\begin{prop}\label{RepG}
Let $G$ be a reductive group whose corresponding Lie algebra $\g$ is simple. Then the set of path components of the $n$-th order $\E_2$-monoidal deformations of $\Rep(G)$ is $\pi_0\E_2\mathrm{CatDef}_{\Rep(G)}^c(k[[x]]/x^{n+1})\simeq \oplus_{n} k$. And for $\E_2$-monoidal formal deformations, we have $\pi_0\E_2\mathrm{CatDef}_{\Rep(G)}^c(k[[x]])\simeq \prod_{\mathbb{Z}}k$.
\end{prop}

\begin{proof}
From Theorem \ref{BG}, we know the $\E_2$-monoidal deformation problem $\E_n\mathrm{CatDef}_{\Rep(G)}^c$ is equivalent to 
$$\Map_{\Alg_{k}^{(4),\mathrm{aug}}}\Big(\mathscr{D}^{(4)}(-),k\oplus\mathrm{fib}\big(Sym_{k}(\g[-2])^G\rightarrow k\big)\Big)$$
where $\mathscr{D}^{(4)}$ is the $\E_{4}$-Koszul duality functor. So we can compute 
\begin{align*}
\E_2\mathrm{CatDef}_{\Rep(G)}^c(k[x]/x^2)&\simeq \Map_{\Alg_{k}^{(4),\mathrm{aug}}}\Big(\mathscr{D}^{(4)}(k[x]/x^2),k\oplus\mathrm{fib}\big(Sym_{k}(\g[-2])^G\rightarrow k\big)\Big)\\
&\simeq \Map_{\Alg_{k}^{(4)}}\Big(\mathrm{Free}_{\E_4}k[-4],\mathrm{fib}\big(Sym_{k}(\g[-2])^G\rightarrow k\big)\Big)\\
&\simeq \Map_{\Mod_{k}}\Big(k[-4], \mathrm{fib}\big(Sym_{k}(\g[-2])^G\rightarrow k\big)\Big)\\
&\simeq \mathrm{fib}\Big(\Map_{\Mod_{k}}(k[-4],Sym_k(\g[-2])^G)\rightarrow \Map_{\Mod_{k}}(k[-4],k)\Big)
\end{align*}
and using long exact sequences, we get 
\[\begin{tikzpicture}[descr/.style={fill=white,inner sep=1.5pt}]
        \matrix (m) [
            matrix of math nodes,
            row sep=1em,
            column sep=2.5em,
            text height=1.5ex, text depth=0.25ex
        ]
        { 0 & \pi_1\E_2\mathrm{CatDef}_{\Rep(G)}^c(k[x]/x^2) & 0 & 0 \\
            & \pi_0 \E_2\mathrm{CatDef}_{\Rep(G)}^c(k[x]/x^2)&  (Sym^2 \g)^G & 0 \\
        };

        \path[overlay,->, font=\scriptsize,>=latex]
        (m-1-1) edge (m-1-2)
        (m-1-2) edge (m-1-3)
        (m-1-3) edge (m-1-4)
        (m-1-4) edge[out=355,in=175] node[descr,yshift=0.3ex] {} (m-2-2)
        (m-2-2) edge (m-2-3)
        (m-2-3) edge (m-2-4);
\end{tikzpicture}\]
Hence $\pi_1\E_2\mathrm{CatDef}_{\Rep(G)}^c(k[x]/x^2)=0$ and from \cite[Prop. 2.16]{Saf21},
$$\pi_0\E_2\mathrm{CatDef}_{\Rep(G)}^c(k[x]/x^2)\cong (Sym^2\g)^G\cong(\wedge^3\g)^G\cong k$$
which means non-trivial first order $\E_2$-monoidal deformations of $\Rep(G)$ are unique up to scaling.

Now we assume we have already computed 
$$\pi_1\E_2\mathrm{CatDef}_{\Rep(G)}^c(k[x]/x^{n+1})=0,\ \ \text{and}\ \ \pi_0 \E_2\mathrm{CatDef}_{\Rep(G)}^c(k[x]/x^{n+1})=\oplus_{n} k$$
Consider the following pullback diagram
\[\begin{tikzcd}
	{k[x]/x^{n+2}} & k \\
	{k[x]/x^{n+1}} & {k\oplus k[1]}
	\arrow[from=1-1, to=1-2]
	\arrow[from=1-1, to=2-1]
	\arrow["\lrcorner"{anchor=center, pos=0.125}, draw=none, from=1-1, to=2-2]
	\arrow[from=1-2, to=2-2]
	\arrow[from=2-1, to=2-2]
\end{tikzcd}\]
and then we get the fiber sequence 
$$\E_2\mathrm{CatDef}_{\Rep(G)}^c(k[x]/x^{n+2})\rightarrow \E_2\mathrm{CatDef}_{\Rep(G)}^c(k[x]/x^{n+1})\rightarrow \E_2\mathrm{CatDef}_{\Rep(G)}^c(k\oplus k[1])$$
where 
\begin{align*}
 \E_2\mathrm{CatDef}_{\Rep(G)}^c(k\oplus k[1])
 & \simeq \Map_{\Alg_{k}^{(4),\mathrm{aug}}}\Big(\mathscr{D}^{(4)}(k\oplus k[1]),k\oplus\mathrm{fib}\big(Sym_{k}(\g[-2])^G\rightarrow k\big)\Big)\\
 & \simeq \Map_{\Alg_{k}^{(4)}}\Big(\mathrm{Free}_{\E_4}k[-5],\mathrm{fib}\big(Sym_{k}(\g[-2])^G\rightarrow k\big)\Big)\\
&\simeq \Map_{\Mod_{k}}\Big(k[-5], \mathrm{fib}\big(Sym_{k}(\g[-2])^G\rightarrow k\big)\Big)\\
&\simeq \mathrm{fib}\Big(\Map_{\Mod_{k}}(k[-5],Sym_k(\g[-2])^G)\rightarrow \Map_{\Mod_{k}}(k[-5],k)\Big)
\end{align*}
and using the long exact sequence, we see
\[\begin{tikzpicture}[descr/.style={fill=white,inner sep=1.5pt}]
        \matrix (m) [
            matrix of math nodes,
            row sep=1em,
            column sep=2.5em,
            text height=1.5ex, text depth=0.25ex
        ]
        { 0 & \pi_2\E_1\mathrm{CatDef}_{\Rep(G)}^c(k\oplus k[1]) & 0 & 0 \\
            & \pi_1\E_1\mathrm{CatDef}_{\Rep(G)}^c(k\oplus k[1]) & (Sym^2\g)^{G} \cong k & 0 \\
            & \pi_0\E_1\mathrm{CatDef}_{\Rep(G)}^c(k\oplus k[1]) & 0 & 0 \\
        };

        \path[overlay,->, font=\scriptsize,>=latex]
        (m-1-1) edge (m-1-2)
        (m-1-2) edge (m-1-3)
        (m-1-3) edge (m-1-4)
        (m-1-4) edge[out=355,in=175] node[descr,yshift=0.3ex] {} (m-2-2)
        (m-2-2) edge (m-2-3)
        (m-2-3) edge (m-2-4)
        (m-2-4) edge[out=355,in=175] node[descr,yshift=0.3ex] {} (m-3-2)
        (m-3-2) edge (m-3-3)
        (m-3-3) edge (m-3-4);
\end{tikzpicture}\]
Next look at the long exact sequence induced by $k[x]/x^{n+2}=\mathrm{pullback}(k[x]/x^{n+1}\rightarrow k\oplus k[1]\leftarrow k)$.
\[\begin{tikzpicture}[descr/.style={fill=white,inner sep=1.5pt}]
        \matrix (m) [
            matrix of math nodes,
            row sep=1em,
            column sep=2.5em,
            text height=1.5ex, text depth=0.25ex
        ]
        { 0 & \pi_1 \E_2\mathrm{CatDef}_{\Rep(G)}^c(k[x]/x^{n+2}) & 0 & (Sym^2 \g)^G\cong k \\
            & \pi_0\E_2\mathrm{CatDef}_{\Rep(G)}^c(k[x]/x^{n+2})&  \oplus _n k & 0 \\
        };

        \path[overlay,->, font=\scriptsize,>=latex]
        (m-1-1) edge (m-1-2)
        (m-1-2) edge (m-1-3)
        (m-1-3) edge (m-1-4)
        (m-1-4) edge[out=355,in=175] node[descr,yshift=0.3ex] {} (m-2-2)
        (m-2-2) edge (m-2-3)
        (m-2-3) edge (m-2-4);
\end{tikzpicture}\]
So by induction we have 
$$\pi_1 \E_2\mathrm{CatDef}_{\Rep(G)}^c(k[x]/x^{n+2})=0,\ \ \text{and}\ \ \pi_0 \E_2\mathrm{CatDef}_{\Rep(G)}^c(k[x]/x^{n+2})\cong \oplus_{n+1} k$$
Because $k[[x]]=\lim_n k[x]/x^n$, from \cite[Lem. 4.23]{BKP} we obtain the equivalence 
$$\E_2\mathrm{CatDef}_{\Rep(G)}^c(k[[x]])=\E_2\mathrm{CatDef}_{\Rep(G)}^c(\lim_n k[x]/x^n)\simeq \lim_n \E_2\mathrm{CatDef}_{\Rep(G)}^c(k[x]/x^n)$$
Using Milnor's sequence in \cite[Thm. 2.1]{Hir15}, we have an exact sequence

\[\begin{tikzpicture}[descr/.style={fill=white,inner sep=1.5pt}]
        \matrix (m) [
            matrix of math nodes,
            row sep=1em,
            column sep=2.5em,
            text height=1.5ex, text depth=0.25ex
        ]
        {   0 & \lim_{n}^{1}\pi_1 \E_2\mathrm{CatDef}_{\Rep(G)}^c(k[x]/x^n)=0 \\
            \pi_0 \lim_{n}\E_2\mathrm{CatDef}_{\Rep(G)}^c(k[x]/x^n)&  \lim_{n}\pi_0\E_2\mathrm{CatDef}_{\Rep(G)}^c(k[x]/x^n) & 0 \\
        };

        \path[overlay,->, font=\scriptsize,>=latex]
        (m-1-1) edge (m-1-2)
        (m-1-2) edge[out=355,in=175] node[descr,yshift=0.3ex] {} (m-2-1)
        (m-2-1) edge (m-2-2)
        (m-2-2) edge (m-2-3);
\end{tikzpicture}\]
so that there is an equivalence
$$\pi_0\E_2\mathrm{CatDef}_{\Rep(G)}^c(k[[x]])\simeq \lim_n (\oplus_{n-1}k)=\prod_{\mathbb{Z}}k$$
\end{proof}

Although here we get $\pi_0\E_2\mathrm{CatDef}_{\Rep(G)}^c(k[[x]])\simeq \prod_{\mathbb{Z}}k$ which seems a bit large, we can actually prove non-trivial formal deformations which induce non-trivial first order deformations, are unique up to equivalences and scaling i.e. $\mathrm{Aut}(k[[x]])$. It is not a special case for $\E_2$-monoidal formal deformations of $\Rep(G)$. 

In \cite[Sec. 5]{Sei02}, Seidel studies $A_\infty$-deformations of $A_\infty$-categories. An $A_\infty$-category $\mathcal{C}$ is a differential graded (DG) category and $A_\infty$-deformations are just to deform it as a plain DG category, which is explained by the deformation theory of plain categories as discussed in previous sections. So such a formal $\E_2$-moduli problem is controlled by the (non-unital) $\E_2$-algebra $HH^*_{\E_{1}}(\C)=\End_{\mathcal{Z}_{\E_0}(\C)}(1)$ i.e. Hochschild cohomology. Sicne $\mathscr{D}^{(2)}(k[x]/x^2)=\mathrm{Free}_{\E_2}(k[-2])$, the set of path components of first order deformations will be $HH^{2}_{\E_{1}}(\C)$. In \cite[p8, Sec. 5]{Sei02}, Seidel states if $HH^{2}_{\E_{1}}(\C)\cong k$ is one dimensional, then a non-trivial formal deformation of $\C$, if it exists, is unique up to equivalences and change of parameter $x\mapsto f(x)$. We can notice in the case of $\Rep(G)$, its first order $\E_1$ and $\E_2$-monoidal deformations both form a one dimensional space. So there should be a more general statement for formal $\E_n$-moduli problems. In the following, we prove this assertion under some assumptions.

\begin{lemma}\label{compare}
Let $R\rightarrow R'$ be a morphism of non-unital $\E_n$-algebras for $n\geq 2$ such that it induces isomorphisms on $H^{n+1}$, $H^n$ and $H^{n-1}$. If $H^{n-1}(R)\cong H^{n-1}(R')=0$ or $H^{n-2}(R)\rightarrow H^{n-2}(R')$ is surjective, then we have an equivalence
$$\pi_0\Map_{\Alg_{k}^{(n)}}(\mathscr{D}^{(n)}(k[[x]]),R)\xrightarrow{\sim}\pi_0\Map_{\Alg_{k}^{(n)}}(\mathscr{D}^{(n)}(k[[x]]),R')$$
\end{lemma}

Note that we can get an augmented $\E_n$-algebra $k\oplus R$ from a non-unital algebra $R$, so here the mapping space $\Map_{\Alg_{k}^{(n)}}(\mathscr{D}^{(n)}(k[[x]]),R)$ actually means $\Map_{\Alg_{k}^{(n),\aug}}(\mathscr{D}^{(n)}(k[[x]]),k\oplus R)$.

\begin{proof}
The proof is similar to the proof above when computing $\pi_0\E_2\mathrm{CatDef}_{\Rep(G)}^c(k[[x]])$. At first, we see 
$$\pi_0\Map_{\Alg_{k}^{(n)}}(\mathscr{D}^{(n)}(k[x]/x^2),R)=\pi_0\Map_{\Alg_{k}^{(n)}}(\mathrm{Free}_{\E_n}(k[-n]),R)\cong H^n(R)\cong H^n(R')$$
and similarly 
\begin{align*}
    &\pi_1\Map_{\Alg_{k}^{(n)}}(\mathscr{D}^{(n)}(k[x]/x^2),R)\cong H^{n-1}(R)\cong H^{n-1}(R')\\
    &\pi_2\Map_{\Alg_{k}^{(n)}}(\mathscr{D}^{(n)}(k[x]/x^2),R)\cong H^{n-2}(R)
\end{align*}
We prove by induction that 
\begin{align*}
    \pi_0\Map_{\Alg_{k}^{(n)}}(\mathscr{D}^{(n)}(k[x]/x^m),R)\cong \pi_0\Map_{\Alg_{k}^{(n)}}(\mathscr{D}^{(n)}(k[x]/x^m),R')\\
    \pi_1\Map_{\Alg_{k}^{(n)}}(\mathscr{D}^{(n)}(k[x]/x^m),R)\cong\pi_1\Map_{\Alg_{k}^{(n)}}(\mathscr{D}^{(n)}(k[x]/x^m),R')
\end{align*}
If $H^{n-1}(R)\cong H^{n-1}(R')\ne 0$ and $H^{n-2}(R)\rightarrow H^{n-2}(R')$ is surjective, we also need to prove by induction that
$$\pi_2\Map_{\Alg_{k}^{(n)}}(\mathscr{D}^{(n)}(k[x]/x^m),R)\rightarrow \pi_2\Map_{\Alg_{k}^{(n)}}(\mathscr{D}^{(n)}(k[x]/x^m),R')$$
is surjective. Now consider the pullback diagram, 
\[\begin{tikzcd}
	{k[x]/x^{m+1}} & k \\
	{k[x]/x^{m}} & {k\oplus k[1]}
	\arrow[from=1-1, to=1-2]
	\arrow[from=1-1, to=2-1]
	\arrow["\lrcorner"{anchor=center, pos=0.125}, draw=none, from=1-1, to=2-2]
	\arrow[from=1-2, to=2-2]
	\arrow[from=2-1, to=2-2]
\end{tikzcd}\]
and then 
$$\Map_{\Alg_{k}^{(n)}}(\mathscr{D}^{(n)}(k[x]/x^{m+1}),R)\simeq \mathrm{fib}\Big(\Map_{\Alg_{k}^{(n)}}(\mathscr{D}^{(n)}(k[x]/x^{m}),R)\rightarrow \Map_{\Alg_{k}^{(n)}}(\mathscr{D}^{(n)}(k\oplus k[1]),R)\Big)$$
where $\mathscr{D}^{(n)}(k\oplus k[1])=\mathrm{Free}_{\E_n}(k[-n-1])$, so that we have
\begin{align*}
&\pi_0\Map_{\Alg_{k}^{(n)}}(\mathscr{D}^{(n)}(k\oplus k[1]),R)\cong H^{n+1}(R)\cong H^{n+1}(R')\\
&\pi_1\Map_{\Alg_{k}^{(n)}}(\mathscr{D}^{(n)}(k\oplus k[1]),R)\cong H^{n}(R)\cong H^{n}(R')\\
&\pi_2\Map_{\Alg_{k}^{(n)}}(\mathscr{D}^{(n)}(k\oplus k[1]),R)\cong H^{n-1}(R)\cong H^{n-1}(R')
\end{align*}
Now we get a long exact sequence
\[\begin{tikzpicture}[descr/.style={fill=white,inner sep=1.5pt}]
        \matrix (m) [
            matrix of math nodes,
            row sep=1em,
            column sep=2.5em,
            text height=1.5ex, text depth=0.25ex
        ]
        { H^{n-2}(R) & \pi_2\Map_{\Alg_{k}^{(n)}}(\mathscr{D}^{(n)}(k[x]/x^{m+1}),R) & \pi_2\Map_{\Alg_{k}^{(n)}}(\mathscr{D}^{(n)}(k[x]/x^{m}),R) & H^{n-1}(R) \\
            & \pi_1\Map_{\Alg_{k}^{(n)}}(\mathscr{D}^{(n)}(k[x]/x^{m+1}),R) & \pi_1\Map_{\Alg_{k}^{(n)}}(\mathscr{D}^{(n)}(k[x]/x^{m}),R)&H^{n}(R) \\
            & \pi_0\Map_{\Alg_{k}^{(n)}}(\mathscr{D}^{(n)}(k[x]/x^{m+1}),R) & \pi_0\Map_{\Alg_{k}^{(n)}}(\mathscr{D}^{(n)}(k[x]/x^{m}),R) & H^{n+1}(R) \\
        };

        \path[overlay,->, font=\scriptsize,>=latex]
        (m-1-1) edge (m-1-2)
        (m-1-2) edge (m-1-3)
        (m-1-3) edge (m-1-4)
        (m-1-4) edge[out=355,in=175] node[descr,yshift=0.3ex] {} (m-2-2)
        (m-2-2) edge (m-2-3)
        (m-2-3) edge (m-2-4)
        (m-2-4) edge[out=355,in=175] node[descr,yshift=0.3ex] {} (m-3-2)
        (m-3-2) edge (m-3-3)
        (m-3-3) edge (m-3-4);
\end{tikzpicture}\]
By five lemma, we get an isomorphism 
$$ \pi_0\Map_{\Alg_{k}^{(n)}}(\mathscr{D}^{(n)}(k[x]/x^{m+1}),R)\xrightarrow{\sim}\pi_0\Map_{\Alg_{k}^{(n)}}(\mathscr{D}^{(n)}(k[x]/x^{m+1}),R')$$
If $H^{n-1}(R)\cong H^{n-1}(R')=0$, we conclude 
$$ \pi_1\Map_{\Alg_{k}^{(n)}}(\mathscr{D}^{(n)}(k[x]/x^{m+1}),R)\xrightarrow{\sim}\pi_1\Map_{\Alg_{k}^{(n)}}(\mathscr{D}^{(n)}(k[x]/x^{m+1}),R')$$
and moreover, they are actually $0$ by induction. If $H^{n-1}(R)\cong H^{n-1}(R')\ne0$ and $H^{n-2}(R)\rightarrow H^{n-2}(R')$ is surjective, by assumption that
$$\pi_2\Map_{\Alg_{k}^{(n)}}(\mathscr{D}^{(n)}(k[x]/x^m),R)\rightarrow \pi_2\Map_{\Alg_{k}^{(n)}}(\mathscr{D}^{(n)}(k[x]/x^m),R')$$
is surjective, the isomorphism of $\pi_1$ is from the five lemma. The four lemma will imply 
$$\pi_2\Map_{\Alg_{k}^{(n)}}(\mathscr{D}^{(n)}(k[x]/x^{m+1}),R)\rightarrow \pi_2\Map_{\Alg_{k}^{(n)}}(\mathscr{D}^{(n)}(k[x]/x^{m+1}),R')$$
is surjective. This completes the induction. Next using the Milnor's sequence for $k[[x]]=\lim k[x]/x^m$, we conclude
$$\pi_0\Map_{\Alg_{k}^{(n)}}(\mathscr{D}^{(n)}(k[[x]]),R)\xrightarrow{\sim}\pi_0\Map_{\Alg_{k}^{(n)}}(\mathscr{D}^{(n)}(k[[x]]),R')$$
\end{proof}

\begin{theorem}\label{unique}
Let $X:\Alg_{k}^{(n), \mathrm{Art}}\rightarrow s\mathbf{Set}$ be a formal $\E_{n}$-moduli problem corresponding to the non-unital $\E_n$-algebra $R$ for $n\geq 2$. If $H^n(R)\cong k$ is one dimensional and $H^{n-1}(R)=H^{n+1}(R)=0$, then $\pi_0X(k[[x]])\simeq \End(k[[x]])$. In particular, formal deformations in $\pi_0X(k[[x]])$ which induce non-trivial first order deformations in $\pi_0X(k[x]/x^2)\cong H^n(R)\cong k$ are unique up to change of parameter i.e. $\mathrm{Aut}(k[[x]])$.
\end{theorem}

\begin{proof}
Because $H^n(R)\cong k$ is one dimensional, by taking a generator we can get a morphism $\mathscr{D}^{(n)}(k[x]/x^2)\rightarrow k\oplus R$ of augmented $\E_n$-algebras which induces an isomorphism 
\[\begin{tikzcd}
	{H^n(\mathscr{D}^{(n)}(k[x]/x^2))=H^n(\mathrm{Free}_{\E_n}(k[-n]))\cong k} & {H^n(R)\cong k}
	\arrow["\sim", from=1-1, to=1-2]
\end{tikzcd}\]
Applying $\mathscr{D}^{(n)}$ to the following pullback
\[\begin{tikzcd}
	{k[x]/x^{3}} & k \\
	{k[x]/x^{2}} & {k\oplus k[1]}
	\arrow[from=1-1, to=1-2]
	\arrow[from=1-1, to=2-1]
	\arrow["\lrcorner"{anchor=center, pos=0.125}, draw=none, from=1-1, to=2-2]
	\arrow[from=1-2, to=2-2]
	\arrow[from=2-1, to=2-2]
\end{tikzcd}\]
we get a pushout diagram
\[\begin{tikzcd}
{\mathscr{D}^{(n)}(k\oplus k[1])} \arrow[d] \arrow[r]           & k \arrow[d] \arrow[rdd, bend left] &           \\
{\mathscr{D}^{(n)}(k[x]/x^2)} \arrow[r] \arrow[rrd, bend right] & {\mathscr{D}^{(n)}(k[x]/x^3)}      &           \\
                                                                &                                    & k\oplus R
\end{tikzcd}\]
where $\mathscr{D}^{(n)}(k\oplus k[1])=\mathrm{Free}_{\E_n}(k[-n-1])$. Since $H^{n+1}(R)=0$, the diagram above is commutative up to homotopy and then there is a map $\mathscr{D}^{(n)}(k[x]/x^3)\rightarrow k\oplus R$ extending $\mathscr{D}^{(n)}(k[x]/x^2)\rightarrow k\oplus R$. This argument implies there should be maps $\mathscr{D}^{(n)}(k[x]/x^{n})\rightarrow k\oplus R$ and hence $\mathscr{D}^{(n)}(k[[x]])\rightarrow k\oplus R$ making the following diagram commutative
\[\begin{tikzcd}
	{\mathscr{D}^{(n)}(k[x]/x^2)} & {\mathscr{D}^{(n)}(k[x]/x^3)} & \cdots & {\mathrm{colim}_m\mathscr{D}^{(n)}(k[x]/x^m)\simeq \mathscr{D}^{(n)}(k[[x]])} \\
	&& {k\oplus R}
	\arrow[from=1-1, to=1-2]
	\arrow[from=1-1, to=2-3]
	\arrow[from=1-2, to=1-3]
	\arrow[from=1-2, to=2-3]
	\arrow[from=1-3, to=1-4]
	\arrow["\exists", dashed, from=1-4, to=2-3]
\end{tikzcd}\]
$\mathscr{D}^{(n)}(k[[x]])$ is actually the free $\E_\infty$-algebra with one generator at degree $n$, viewed as an $\E_n$-algebra. So the morphism $\mathscr{D}^{(n)}(k[[x]])_{>0}\rightarrow R$ of (non-unital) $\E_n$-algebras satisfies the Lemma above and we get equivalences
\begin{align*}
    \pi_0X(k[[x]])&=\pi_0\Map_{\Alg_{k}^{(n),\aug}}(\mathscr{D}^{(n)}(k[[x]]),k\oplus R)\\
    &\simeq\pi_0\Map_{\Alg_{k}^{(n),\aug}}(\mathscr{D}^{(n)}(k[[x]]),\mathscr{D}^{(n)}(k[[x]]))\\
    &\simeq \pi_0\Map_{\Alg_{k}^{(n),\aug}}(k[[x]],k[[x]])\\
    &\simeq \End_{\mathbf{CAlg}_{k}^{\heartsuit,\aug}}(k[[x]])
\end{align*}
So that there is a natural action of $\mathrm{Aut}(k[[x]])$ on $\pi_0X(k[[x]])$ given by composition. Then formal deformations, viewed as endomorphisms of $k[[x]]$ of the form $x\mapsto a_1x+a_2x^2+\cdots$ such that $a_1\ne0$, are just those inducing non-trivial first order deformations. For any two such formal deformations, there exists an automorphism of $k[[x]]$ sending one to the other. This means these formal deformations are unique up to change of parameter.
\end{proof}

\begin{corollary}
Let $G$ be a reductive group whose corresponding Lie algebra $\g$ is simple. Then non-trivial $\E_1$ and $\E_2$-monoidal formal deformations of $\Rep(G)$ inducing non-trivial first order deformations are unique up to equivalences and change of parameter.
\end{corollary}

\subsection{Deformations by Factorization Homology}
Let $\Mfld_n$ be the $\infty$-category of smooth topological manifolds of dimension $n$ such that for two manifolds $M$ and $N$, its mapping space is $\mathrm{Emd}(M,N)$ of embeddings equipped with the compact-open $C^\infty$ topology. If we suppose the tensor product is the disjoint union, then $\Mfld_n$ will be a symmetric monoidal $\infty$-category. Let $\Disk_n$ be the sub-$\infty$-category of $\Mfld_n$ consisting of disjoint unions of $\mathbb{R}^n$. We are interested in the $\infty$-operad $\Disk_{n/M}$ whose objects are embeddings from some open disk or $\mathbb{R}^n$ to $M$. If $M$ is contractible, \cite[Example 5.4.2.15]{HA} shows $\Disk_{n/M}$ is equivalent to $\mathbb{E}_{n}^{\otimes}$ the $\E_n$-operad.

We can also consider framed manifolds. Fix a space $B$ and a map $B\rightarrow BO(n)\simeq BGL(n)$. For a manifold $M$ of dimension $n$, there is a natural map $\tau_M:M\rightarrow BO(n)$ corresponding to its tangent bundle. So a \textit{$B$-framed $n$-manifold} is actually a lifting of $\tau_M$ i.e. the following commutative diagram in $s\Set$
\[\begin{tikzcd}
	& B \\
	M & {BO(n)}
	\arrow[from=1-2, to=2-2]
	\arrow[from=2-1, to=1-2]
	\arrow["{\tau_M}"', from=2-1, to=2-2]
\end{tikzcd}\]
The space of morphisms denoted by $\mathrm{Emd}^{B}(M,N)$ consists of embeddings preserving such $B$-framed structures. So that we get $\infty$-categories $\Mfld_{n}^{B}$ and $\Disk_{n}^{B}$. If we let $B\rightarrow BO(n)$ be $*\xrightarrow{\{\mathbb{R}^n\}}BO(n)$, \cite[Remark 2.28]{AF20} shows the resulting $\infty$-operad denoted by $\Disk_{n}^{fr}$ is equivalent to the $\E_n$-operad. Every framing in it gives a trivialization of the tangent bundle of some $M$.

Let $\Mfld_{n}^{\partial}$ be the $\infty$-category of $n$-manifolds possibly with boundary. Then $\Disk_{n}^{\partial}$ should consist of disjoint unions of $\mathbb{R}^n$ and $\mathbb{R}_{\geq 0}\times \mathbb{R}^{n-1}$. It is stated in \cite[Remark 2.37]{AF20} that together with the framing structure, $\Disk_{n}^{\partial,fr}$ is equivalent to the \textit{Swiss cheese operad}. A $\Disk_{n}^{\partial,fr}$-algebra is a pair $(A,V)$ of the image $(\mathbb{R}^n,\mathbb{R}_{\geq 0}\times \mathbb{R}^{n-1})$ such that $A$ is an $\E_n$-algebra and $V$ is an $\E_{n-1}$-algebra with an action from $A$ \cite[Thm. 3.3]{Vor98}. Following \cite{BZBJ}, here we only consider the case $V=A$.

One can also think about oriented manifolds and oriented embeddings. In this way, we get $\Mfld_{n}^{or}$ and $\Disk_{n}^{or}$. Algebras on $\Disk_{n}^{or}$ are \textit{framed $\E_n$-algebras}. For example, a framed $\E_2$-monoidal category is actually a balanced braided monoidal category.

\begin{definition}
    \normalfont
Let $\C$ be a sifted-completed symmetric monoidal $\infty$-category which means $\C$ admits small sifted colimits and the tensor product preserves small sifted colimits. For any $\E_n$-algebra $A$ in $\C$, the \textit{factorization homology with coefficients in $A$} defines an operadic left Kan extension along the inclusion $\Disk_{n}^{fr}\rightarrow \Mfld_{n}^{fr}$, whose value on a framed $n$-manifold $M$ is 
$$\int_{M}A:=\colim(\Disk_{n/M}^{fr}\rightarrow \Disk_{n}^{fr}\xrightarrow{A}\C)$$
For a framed $\E_n$-algebra $A$, $\int_M A$ is defined similarly for an oriented $n$-manifold $M$.
\end{definition}

In this subsection, we are more interested in the factorization homology of categories. Let $G$ be a reductive group. Suppose $\Rep_\hbar(G)$ is an $\E_2$-monoidal deformation quantization of $\Rep(G)$, e.g. the category of representations of the quantum group $U_{\hbar} \g$. In \cite{BZBJ}, for any framed surface $S$, the authors regard $\int_S \Rep_\hbar(G)$ as a deformation quantization of $\QCoh(\mathrm{LocSys}_G(S))$. It is actually motivated by the following proposition, which is a special case of \cite[Cor. 4.12]{BFN10}.

\begin{prop}
Let $G$ be a reductive group. Suppose $S$ is a topological surface with finitely many connected components. Then we have an equivalence
$$\int_S\Rep(G)\simeq \QCoh(\mathrm{LocSys}_G(S))$$
\end{prop}

\begin{remark}
    \normalfont
In addition to $\Pr^L$ and $\Pr^{St}$, we can also let the target category $\C$ be $\Pr^{L,c}$ consisting of compactly generated presentable $\infty$-categories with compact small colimits preserving functors or its sub-$\infty$-category $\mathsf{Groth}^{\mathrm{cg}}_{\infty}$ following \cite[Prop. C.6.2.1]{SAG}. 

It is shown in \cite[Example 5.5.2.4]{HA} that if $\C$ is a symmetric monoidal $\infty$-category such that it admits small colimits and the tensor product preserves small colimits  separately in each variable, then for any small $\infty$-operad $\mathcal{O},$ $\Alg_{\mathcal{O}}(\C)$ will be sifted-complete symmetric monoidal $\infty$-category having small colimits. So that in particular, $\Alg_{\E_n}(\Pr^{L,c})$ is symmetric monoidal and sifted-complete.
\end{remark}

In \cite{KKMP}, the authors have studied the compatibility of factorization homology and deformations. In \cite[Thm. 6.8]{KKMP}, they prove that in a special case of $k[[\hbar]]\rightarrow k[\hbar]/\hbar^2$, the tensor product $-\otimes_{k[[\hbar]]}k[\hbar]/\hbar^2$ should commute with factorization homology. It means for a framed $m$-manifold $M$ and an $\E_n$-monoidal $k[[\hbar]]$-linear deformation $\C_\hbar$ of $\C$, we have an equivalence
$$(\int_M \C_\hbar)\otimes_{k[[\hbar]]}k[\hbar]/\hbar^2\simeq \int_M(\C_\hbar\otimes_{k[[\hbar]]}k[\hbar]/\hbar^2)$$
More generally, we can prove the following theorem at the $\infty$-categorical level.

\begin{theorem}
Suppose $\C$ is an $\E_n$-monoidal category in $\Pr_{k}^{L,c}$ and $M$ is a framed $m$-manifold with $m\leq n$. Then there is a morphism $\E_n\mathrm{CatDef}_{\C}^{c}\rightarrow \E_{n-m}\mathrm{CatDef}_{\int_M\C}^{c}$ of functors $\mathbf{CAlg}_{k}^{\mathrm{Art}}\rightarrow s\Set$ given by the factorization homology over $M$.
\end{theorem}

\begin{proof}
For a map of commutative Artin algebras $A\rightarrow B$, we show the following diagram is commutative.
\[\begin{tikzcd}
	{\Alg_{\E_n}(\Pr_{A}^{L,c})\simeq \Alg_{\E_{m}}(\Alg_{n-m}(\Pr_{A}^{L,c}))} & {\Alg_{\E_n}(\Pr_{B}^{L,c})\simeq \Alg_{\E_{m}}(\Alg_{n-m}(\Pr_{B}^{L,c}))} \\
	{\Alg_{n-m}(\Pr_{A}^{L,c})} & {\Alg_{n-m}(\Pr_{B}^{L,c})}
	\arrow[from=1-1, to=1-2]
	\arrow["{\int_M}"', from=1-1, to=2-1]
	\arrow["{\int_M}", from=1-2, to=2-2]
	\arrow["{F=-\otimes_{A}B}"', from=2-1, to=2-2]
\end{tikzcd}\]
From the remark above, we know $\Alg_{n-m}(\Pr_{A}^{L,c})$ and $\Alg_{n-m}(\Pr_{B}^{L,c})$ are symmetric monoidal and sifted-complete, so they can be the target of factorization homology. Because the tensor product functor $-\otimes_A B:\Pr_{A}^{L,c}\rightarrow \Pr_{B}^{L,c}$ preserves colimits and sifted colimits in $\Alg_{n-m}(\Pr_{A}^{L,c})$ can be tested in $\Pr_{A}^{L,c}$ \cite[Prop. 3.2.3.1]{HA}, the functor $F$ in the diagram above should preserve sifted colimits. Then from \cite[Prop. 5.5.2.17]{HA}, for every $\E_n$-monoidal $A$-linear $\infty$-category $\D\in \Alg_{\E_n}(\Pr_{A}^{L,c})$, we have an equivalence
$$(\int_M \D)\otimes_A B\simeq \int_M(\D\otimes_A B) $$
We can suppose $\D=\C_A$ is an $\E_n$-monoidal deformation of $\C$ over $A$, then $\int_M \C_A$ is an $\E_{n-m}$-monoidal $A$-linear deformation of $\int_M \C$.
\end{proof}

\bibliographystyle{alpha}
\bibliography{ref} 

\end{document}